\documentclass[11pt,a4]{amsart}

\usepackage{comment}
\usepackage{mathtools,amsthm,amssymb,graphicx,fancyhdr,mathtools}
\usepackage{color}
\usepackage[top=1in, bottom=1in, left=1in, right=1in]{geometry}
\usepackage{enumerate}
\usepackage[titletoc,title]{appendix}
\usepackage{multirow}

\newtheorem{theorem}{Theorem}[section]

\newtheorem{definition}{Definition}[section]
\newtheorem{lemma}{Lemma}[section]
\newtheorem{corollary}{Corollary}[section]

\newtheorem{remark}{Remark}[section]
   \newtheoremstyle{example}{\topsep}{\topsep}%
     {}%         Body font
     {}%         Indent amount (empty = no indent, \parindent = para indent)
     {\bfseries}% Thm head font
     {}%        Punctuation after thm head
     {\newline}%     Space after thm head (\newline = linebreak)
     {\thmname{#1}\thmnumber{ #2}\thmnote{ #3}}%         Thm head spec

   \theoremstyle{example}

\newcommand{\Prob}{\mathbb{P}}
\newcommand{\Poly}{\mathcal{P}}

\newcommand{\bfP}{\mathbf{P}}
\newcommand{\bfQ}{\mathbf{Q}}
\newcommand{\tbfQ}{\tilde{\mathbf{Q}}}
\newcommand{\bfU}{\mathbf{U}}
\newcommand{\bfV}{\mathbf{V}}
\newcommand{\tbfV}{\tilde{\mathbf{V}}}
\newcommand{\bfL}{\mathbf{L}}
\newcommand{\bfG}{\mathbf{G}}

\newcommand{\hatX}{\widehat{X}}
\newcommand{\hatY}{\widehat{Y}}
\newcommand{\hatP}{\widehat{P}}
\newcommand{\hatQ}{\widehat{Q}}
\newcommand{\hatU}{\widehat{U}}
\newcommand{\hatV}{\widehat{V}}
\newcommand{\hatIp}{\widehat{\Lambda}_{\phi}}
\newcommand{\hatLambda}{\widehat{\Lambda}}
\newcommand{\hatphi}{\varphi^{\widehat{X}}}
\newcommand{\hatrho}{\varphi^{\widehat{Y}}}

\newcommand{\ttE}{\mathtt{E}}
\newcommand{\Co}{\mathtt{C}_0(\R_+)}

\newcommand{\CbE}{\mathtt{C}_b(E)}

\newcommand{\Ci}{\mathtt{C}^{\infty}}
\newcommand{\Bb}{\mathtt{B}_b}
\newcommand{\rmL}{{\rm{L}}}

\newcommand{\Lnu}{{\rm{L}}^2(\frakm)}
\newcommand{\Lg}{{\rm{L}}^2(m)}
\newcommand{\Lm}{{\rm{L}}^2(\mbar)}
\newcommand{\rmd}{\mathrm{d}}
\newcommand{\E}{\mathbb{E}}

\newcommand{\R}{{\mathbb{R}}}
\newcommand{\Rcal}{{\mathcal{R}}}

\newcommand{\B}{{\mathbf{B}}}

\newcommand{\C}{{\mathbb{C}}}

\newcommand{\N}{\mathcal{N}}

\newcommand{\Ip}{\Lambda_{\phi}}

\newcommand{\Lag}{\mathcal{L}}
\newcommand{\T}{\mathcal{T}}
\newcommand{\Vp}{V_{\psi}}
\newcommand{\Pbar}{\overline{P}}
\newcommand{\mbar}{\overline{m}}
\newcommand{\Qbar}{\overline{Q}}
\newcommand{\Xbar}{\overline{X}}
\newcommand{\Ybar}{\overline{Y}}

\newcommand{\M}{\mathcal{M}}
\newcommand{\Z}{\mathtt{Z}}

\newcommand{\e}{\mathbf{e}}

\newcommand{\ttP}{\breve{P}}
\newcommand{\frakr}{\mathfrak{r}}
\newcommand{\frakb}{\mathfrak{b}}
\newcommand{\frakm}{\mathfrak{m}}
\newcommand{\frakn}{\mathfrak{n}}
\newcommand{\frakc}{\mathfrak{c}}
\newcommand{\frakl}{\mathfrak{l}}

\newcommand{\D}{\mathcal{D}}

\newcommand{\Fcal}{\mathcal{F}}

\newcommand{\wbar}{\breve{w}}

\newcommand{\eqindist}{\,{\buildrel d \over =}\,}
\newcommand{\eqae}{\,{\buildrel a.e. \over =}\,}

\newcommand{\subsetdense}{\subseteq_d}

\numberwithin{equation}{section}
\author{P. Patie}\thanks{This work was partially supported by  NSF Grant DMS-1406599 and ARC IAPAS, a fund of the Communaut\'ee francaise de Belgique. The first and third  author are grateful  for the hospitality of  the  LMA at the UPPA, where part of this work was completed.}
\address{School of Operations Research and Information Engineering, Cornell University, Ithaca, NY 14853.}
\email{	pp396@orie.cornell.edu}

\author{M. Savov} \thanks{The second author also acknowledges the support of the project MOCT, which has received funding from the European Union’s
	Horizon 2020 research and innovation programme under the Marie Sklodowska-Curie grant
agreement No 657025.}
\address{Institute of Mathematics and informatics,  Bulgarian academy of sciences, Akad.  Georgi Bonchev street
	Block 8, Sofia 1113.}
\email{mladensavov@math.bas.bg}

\author{Y. Zhao}
\address{School of Operations Research and Information Engineering, Cornell University, Ithaca, NY 14853.}
\email{	yz645@cornell.edu}
\title{Intertwining, excursion theory and Krein theory of strings for non-self-adjoint Markov semigroups}
\begin{document}
\maketitle

\begin{abstract}
 In this paper, we start by showing that the intertwining relationship  between  two minimal Markov semigroups acting on Hilbert spaces implies that any recurrent extensions, in the sense of It\^o, %at some regular reference point,
 of these semigroups satisfy the same intertwining identity.  Under mild additional assumptions on the intertwining operator, we prove that the converse also holds.  This connection, which relies on the representation of excursion  quantities as developed by Fitzsimmons and Getoor \cite{Fitzsimmons_06},  enables us to give an interesting probabilistic interpretation of intertwining relationships between  Markov semigroups via excursion theory:  two such recurrent extensions that intertwine share, under an appropriate normalization, the same local time at the boundary point.   Moreover, in the case when one of the (non-self-adjoint) semigroup intertwines with the one of a quasi-diffusion, we obtain  an extension of Krein's theory of strings by  showing that its densely defined spectral measure is absolutely continuous with respect to the measure appearing in the Stieltjes representation of the Laplace exponent of the inverse local time. Finally, we illustrate our results with the class of positive self-similar Markov semigroups and also the reflected generalized Laguerre semigroups. For the latter,  we obtain their spectral decomposition and provide, under some conditions, a perturbed spectral gap estimate for its convergence to equilibrium.% their convergence to e  ansome examples of this intertwining relationship by studying the (non-self-adjoint) self-similar and reflected generalized Laguerre semigroups, and show that these semigroups satisfy the weak-Krein property by deducing their spectral expansions.
\end{abstract}
\section{Introduction}
The famous problem ``Can we hear the shape of a drum?" raised by Kac \cite{Kac_Drum} in 1966 has attracted much attention in the past decades. The question asks that whether one  can determine a planar region $\Omega\subseteq \R^2$, up to geometric congruence, from the knowledge of all the eigenvalues of the  problem
\begin{eqnarray*}
\frac12 \Delta u+\lambda u &=& 0 \quad \textnormal{ on }\Omega,
\end{eqnarray*}
where $\Delta$ is the Laplacian operator, with Dirichlet or Neumann boundary conditions. In other words, if we consider the triplet $(\Delta,\Omega,(\lambda_n)_{n\geq 0})$ where $(\lambda_n)_{n\geq 0}$ represents the sequence of eigenvalues of $\Delta$ on $\Omega$, then Kac's problem asks if $\Omega$ can be determined by providing $(\lambda_n)_{n\geq 0}$. It was not until 1992 that Gordon, Webb and Wolpert \cite{GWW_1992} answered this question negatively by constructing a counterexample with two non-congruent planar domains $\Omega_1$ and $\Omega_2$ which are \textit{isospectral}, that is, the sequence of eigenvalues of $\Delta$ on these domains coincide, counted with multiplicities. These domains are the first planar instances of non-isometric, isospectral,  compact connected Riemannian manifolds that were  previously enunciated by Sunada \cite{Sunada_1985} in the context of the Laplace Beltrami operator.
An equivalent formulation of Kac's problem can be described as follows. Writing $(P_t^{\Omega_j})_{t\geq 0}$, $j=1,2$, the semigroups generated by $\Delta|_{\Omega_j}$ on $\rmL^2(\Omega_j)$, and assuming that there exists a unitary operator $\Lambda:\rmL^2(\Omega_2)\mapsto \rmL^2(\Omega_1)$ such that
\[ P_t^{\Omega_1} \Lambda f = \Lambda P_t^{\Omega_2}f\]
for all $f\in \rmL^2(\Omega_2)$, then does it follow that $\Omega_1$ and $\Omega_2$ are congruent? This idea was first exploited by B\'erard \cite{Ber_1992,Berard_II} who reconsidered Sunada's isospectral problem by providing an explicit transplantation map, that is an intertwining operator which is an unitary isomorphism, which carries each  eigenspace in $\rmL^2(\Omega_2)$ into the corresponding eigenspace in  $\rmL^2(\Omega_1)$.  In addition, Arendt \cite{Arendt_02} (resp.~ Arendt et al.~\cite{Arendt_2012}) showed that for subdomains of $\R^N$ (resp.~for manifolds), if the intertwining operator is order isomorphic, that is, $\Lambda$ is linear, bijective, and $f\geq 0\:a.e.~\Leftrightarrow \Lambda f\geq 0\: a.e.$, then $\Omega_1$ and $\Omega_2$ are congruent, offering a positive answer to Kac's problem.
Furthermore, Arendt et al.~\cite{Arendt_2014} considered a more general setting by studying isospectrality of the Dirichlet or Neumann  type semigroups associated to elliptic operators, including non-self-adjoint ones, by means of the concept of similarity, which is an intertwining relationship with $\Lambda$ a bounded operator with a bounded inverse  from the Hilbert space $\rmL^2(\Omega_1)$ to $\rmL^2(\Omega_2)$. Note that the similarity relation between their corresponding semigroups is equivalent to the isospectral property in the case of Laplacians, but, in general, a stronger property  for non-self-adjoint operators. On the other hand, for $\Omega_i \subset \R^2$, they also showed that it is impossible to have a similarity transform that simultaneously intertwins Dirichlet and Neumann operators on $\Omega_1$ and $\Omega_2$, and therefore there does not exist a similarity transform that intertwins elliptic operators with Robin boundary conditions.

In this paper, we reconsider these problems from another perspective. More specifically, we consider the intertwining relationship
\begin{equation} \label{eq:intertwin_intro}
P_t \Lambda f = \Lambda Q_tf
\end{equation}
where  $P=(P_t)_{t\geq 0}$ and $Q=(Q_t)_{t\geq 0}$ are two Markov semigroups defined on $\rmL^2(\frakm)=\rmL^2(E,\frakm)$ and $\rmL^2(m)=\rmL^2(E,m)$, respectively, with $(E,\mathcal{E})$ a Lusin state space which contains a point $b\in E$ which is regular for the two semigroups, $m, \frakm$ two measures, and $\Lambda: \rmL^2(m) \mapsto \rmL^2(\frakm)$ is merely a densely defined closed and one-to-one operator.
In other words, compared to Kac's framework, we  are interested in a (weak)  isospectrality  from an analytical  viewpoint rather than a geometric one: while  the state space remains the same we consider different operators acting on this domain that intertwine in a weak sense. We emphasize that the fact that we do not require a similarity relation between the operators may imply that their spectrum differ drastically.

The first issue we investigate is to understand whether in our set up  the intertwining relation is stable under any modification of  the boundary conditions. For instance,  is that  possible that there exists an operator that links simultaneously the Dirichlet and Neumann operators, providing an opposite answer to the one obtained in \cite{Arendt_2014} for identical operators acting on different planar domains? We shall show that indeed if two Dirichlet semigroups intertwin (in the sense given above) then any of their recurrent extensions in the sense of It\^o,  are also linked with the same operator. This includes for instance the case of Neumann boundary condition, but also reflecting type condition with a jump and sticky boundary conditions and a mixture of them. We carry on by providing sufficient conditions for the reverse claims to hold.

 We proceed by studying the following question. Can one provide a probabilistic  interpretation of intertwining relationships between Markov semigroups?  This is a natural and fundamental question as  this type of commutation relationships appears  in various issues in recent studies of stochastic processes, see e.g.~\cite{PS_intertwining,PS_Cauchy_Problem,PS_spectral,Diaconis_90,Fill_2009,PS_intertwining}.   We show that when two Dirichlet semigroups intertwin then any of its recurrent extension share, under an appropriate normalization,  the same  local time at the regular boundary point. Indeed we prove that the law of their inverse local time which is, from the general theory of Markov processes, a subordinator, is characterized by the same Bernstein function.
This has the nice pathwise interpretation that the intertwining Markov processes behave the same at a common regular boundary point, but, of course, have different behavior elsewhere.

%In other words, compared to Kac's framework, we fix the domain but consider different operators that intertwine in a weak sense. The first issue we investigate is to understand whether, in our setup, it is possible for both $P$ and $Q$ to have Robin boundary conditions at the point $b$. This issue complements the observation in Arendt et.~al.~\cite{Arendt_2014} in the sense that, under Kac's framework where $P^{\Omega_1}$ and $P^{\Omega_2}$ are semigroups corresponding to the same Robin operator but acting on different domains in $\R^2$, then no such intertwining relationship exists between them. Our approach starts by considering the semigroups $P$ and $Q$ killed at $b$, denoted by $P^{\dag}=(P^{\dag}_t)_{t\geq 0}$ and $Q^{\dag}=(Q^{\dag}_t)_{t\geq 0}$ respectively, both of which have Dirichlet boundary condition at $b$, and prove that if $P^{\dag}_t \Lambda f = \Lambda Q^{\dag}_t f$
%with $\Lambda$ satisfying some mild conditions, then any recurrent extensions of $P^{\dag}$ and $Q^{\dag}$ also intertwine through the same operator $\Lambda$, including those extensions of Neumann or Robin type.  Moreover, since intertwining relationship have been studied intensively in probabilistic literature, see \cite{Diaconis_90,Fill_2009,PS_intertwining}, we proceed by deducing a probabilistic interpretation of the intertwining relationship between Markov semigroups. In particular, we show that if $P$ and $Q$ satisfy \eqref{eq:intertwin_intro} with a common regular point $b$, then they also share the same inverse local time at $b$ if the local times are normalized appropriately.
Next, we recall that the inverse local time of a quasi-diffusion also plays an important role in Krein's spectral theory of strings, since it contains information about the spectrum of the quasi-diffusion process killed at the boundary. Therefore, the question arises naturally that whether one can, through an intertwining relation with the semigroup of a quasi-diffusion, derive a similar result for non-diffusions. We answer this question positively by showing that if $P$ and $Q$ satisfy relation \eqref{eq:intertwin_intro} with $Q$ being the semigroup of a quasi-diffusion, then the Laplace exponent of the inverse local time of the (non-diffusion) Markov process corresponding to $P$ also admits a Stieltjes representation, and the (densely defined) spectral measure of the killed semigroup of $P$ is absolutely continuous with respect to the measure appearing in this Stieltjes representation. This defines a weaker version of Krein's property, which can be seen as an extension to Krein's theory to non-diffusions.

The rest of this paper is organized as follows. After this current section of introduction and basic setups, we start in Section 2 by stating our main theorem and its three corollaries, which give results on the intertwining of semigroups of recurrent extensions, excursion theory and Krein's theory of strings. We prove these results in Section 3. In Section 4, we provide two classes of semigroups which serve as examples for such intertwining relationship. In particular, we study the classes of positive self-similar semigroups and reflected generalized Laguerre semigroups, and show that these (non-self-adjoint) semigroups intertwine with the Bessel semigroup and (classical) Laguerre semigroup respectively. We also deduce the expression for the Laplace exponents of their inverse local times. For a reflected generalized Laguerre semigroup, we also obtain in Section 4 its spectral expansion under some conditions, and derive its rate of convergence to equilibrium, which follows a perturbed spectral gap estimate.

\subsection{Preliminaries} \label{sec:prelim}
 Let $(E,\mathcal{E})$ be a Lusin state space, with $\Bb(E)$ (resp.~$\Bb^+(E)$) denote the space of bounded real-valued  (resp.~bounded real-valued and non-negative) measurable functions on $E$, and $\CbE$ denote the space of bounded continuous functions on $E$. Let $X=(X_t)_{t\geq 0}$ (resp.~$Y=(Y_t)_{t\geq 0}$) defined on a filtered probability space $(\Omega,\Fcal,(\Fcal_t)_t\geq 0,\Prob)$ be a strong Markov process on $E$, which is assumed to have an infinite lifetime, and let $P=(P_t)_{t\geq 0}$ (resp.~$Q=(Q_t)_{t\geq 0}$) denote its corresponding Borel right semigroup, that is, $P_tf(x)=\E_x[f(X_t)]$ (resp.~$Q_tf(x)=\E_x[f(Y_t)]$) for $f\in \Bb(E)$, where $\E_x$ denote the expectation under measure $\Prob_x(X_0=x)=1$ (resp.~$\Prob_x(Y_0=x)=1$). We also assume that for any $f\in \CbE$ (resp.~$\Bb(E)$) and $x\in E$, the mappings
\begin{equation}\label{eq:stoc_cont}
t\mapsto P_tf(x) \textnormal{ and } t\mapsto Q_tf(x) \textnormal{ are continuous (resp.~Borel)},
\end{equation}
and we recall that condition \eqref{eq:stoc_cont} also means that $P_t$ and $Q_t$ are \textit{stochastically continuous}, see e.g.~\cite[Definition 5.1]{GDaPrato_analysis}. We further suppose that $b\in E$ is a regular point for itself, that is $\Prob_b(T^X_b=0)=\Prob_b(T^Y_b=0)=1$, where $T^X_b=\inf\{t>0;\:X_t=b\}$ is the hitting time of $b$ for process $X$, and $T^Y_b$ is defined similarly. Let $X^{\dag}=(X^{\dag}_t)_{t\geq 0}=(X_t;\:0\leq t \leq T^X_b)$ be the process $X$ killed when it hits $b$, after which it is sent to the cemetery point $\Delta$, where we adopt the usual convention that a real-valued function $f$ on $E$ can be extended to $\Delta$ by $f(\Delta)=0$.  We also let $P^{\dag}=(P^{\dag}_t)_{t \geq 0}$ denote the semigroup of $X^{\dag}$, i.e.~$P^{\dag}_tf=\E_x[f(X_t);\:t<T^X_b]$, and we define the process $Y^{\dag}=(Y^{\dag}_t)_{t\geq 0}$ along with its semigroup $Q^{\dag}=(Q^{\dag}_t)_{t\geq 0}$ in  a similar fashion. Next, let $U_qf = \int_0^{\infty}e^{-qt}P_tf dt$ and $U_q^{\dag}f=\int_0^{\infty}e^{-qt}P^{\dag}_tfdt$ be the resolvents of $P$ and $P^{\dag}$, respectively, and, $V_q$ and $V^{\dag}_q$ be the resolvents of $Q$ and $Q^{\dag}$.

We now assume that there exists an excessive measure $\frakm$ (resp.~$m$) on $(E,\mathcal{E})$ for the semigroup $P$ (resp.~$Q$), i.e.~$\frakm$ (resp.~$m$) is a $\sigma$-finite measure and $\frakm P_t\leq \frakm$ (resp.~$mQ_t\leq m$) for all $t>0$, and in particular, when $\frakm P_t = \frakm$ (resp.~$mQ_t=m$), $\frakm$ (resp.~$m$) is an invariant measure. Then a standard argument, see \cite[Theorem 5.8]{GDaPrato_analysis}, indicates that $P$ extends uniquely into a strongly continuous semigroup on $\rmL^2(\frakm)$, which is the weighted Hilbert space
\[\Lnu = \{ f:E\rightarrow \R \textnormal{ measurable }; \|f\|_{\frakm}=\int_E f^2(x)\frakm(dx)<\infty\}\]
endowed with the norm $\|\cdot \|_{\frakm}$ (when there is no confusion and for sake of simplicity, If $\frakm$ is absolutely continuous, we also use $\frakm$ to denote its density and write $\Lnu$ the Hilbert space with weight $\frakm(x)dx$.) Similarly, $Q$ also admits a strongly continuous extension to $\rmL^2(m)$. Note that since $\frakm P^{\dag}_t \leq \frakm P_t \leq \frakm$, $\frakm$ is also an excessive measure for $P^{\dag}$, hence $P^{\dag}$ can also be uniquely extended to a strongly continuous semigroup on $\rmL^2(\frakm)$. Similar results holds for $Q^{\dag}$ as well.

Now let us follow the construction as described in Getoor \cite{Getoor_1999}  to observe that there exists a left-continuous $\hatX=(\hatX_t)_{t\geq 0}$ under the probabilty measure $\widehat{\Prob}_x$, which is the dual process of $X$ with respect to $\frakm$, and is moderate Markov. Note that the measures $(\widehat{\Prob}_x)_{x\in E}$ are only determined modulo an $\frakm$-polar set. Let $\hatP_tf=\widehat{\Prob}_x[f(\hatX_t)]$ denote the moderate Markov dual semigroup associated with $\hatX$ and $\hatU_q$ be the resolvent, then $\hatP$ and $\hatU_q$ are linked to $P$ and $U_q$ via the duality formula
\[(P_tf,g)_{\frakm}=(f,\hatP_t g)_{\frakm}, \quad (U_qf,g)_{\frakm}=(f,\hatU_q g)_{\frakm}\]
for each $f,g\in\Bb(E), q>0, t\geq 0$, where throughout we denote
\begin{equation}\label{defn_bracket}
(f,g)_{\frakm}=\int_E f(x)g(x) \frakm(dx)
\end{equation}
whenever this integral exists.

%Let us write
%  and we denote its infinitesimal generator by $\bfL^{\dag}$ with domain $\D(\bfL^{\dag})$. The same argument goes for $Q^{\dag}$ and we use the obvious notations $\bfG^{\dag}$ and $\D(\bfG^{\dag})$ for its infinitesimal generator and domain.

Because $b$ is a regular point, the singleton $\{b\}$ is not semipolar and there exists a local time $\frakl^X$ at $b$, which is a positive continuous additive functional of $X$, increasing only on the visiting set $\{t\geq 0;\: X_t=b\}$. We mention that $\frakl^X$ is uniquely determined up to a multiplicative constant. The inverse local time $\tau^X=(\tau^X_t)_{t\geq 0}$ is the right continuous inverse of $\frakl^X$, i.e.~
\[\tau^X_t=\inf\{s>0;\:\frakl^X_s>t\}, \quad t\geq 0.\]
It is a standard argument that under the law $\Prob_x$, $\tau^X$ is a strictly increasing subordinator and therefore for any $q>0$,
\[\E_x[e^{-q\tau^X_t}]=e^{-t\Phi_X(q)},\]
where $\Phi_X(q)$ is the Laplace exponent of $\tau^X$ and admits the following L\'evy-Khintchin representation
\begin{equation} \label{eq:Levy_Khin_Phi}
\Phi_X(q)=\delta_X +q \gamma_X+\int_0^{\infty}(1-e^{-qr})\mu_X(dr),
\end{equation}
with $\delta_X=\lim_{q\rightarrow 0} \Phi_X(q)$ is the so-called killing parameter, $\gamma_X=\lim_{q\rightarrow \infty}\frac{\Phi_X(q)}{q}$ is the so-called elasticity parameter, and $\mu_X$ is the L\'evy measure of $\tau^X$, that is a $\sigma$-finite measure concentrated on $(0,\infty)$ satisfying $\int_0^{\infty}(1\wedge y)\mu_X(dy)<\infty$. Furthermore, we follow \cite[Chapter X, Section 2]{RY_Cont_Martingale} to define the so-called Revuz measure $\mathfrak{R}_{\frakl^X}$ for local time $\frakl^X$ as
\[\mathfrak{R}_{\frakl^X}f=\lim_{t\rightarrow 0}\frac1t \E_{\frakm}\left[\int_0^t f(X_s)d\frakl^X_s\right],\]
which, in the case when $\frakm$ is an invariant measure, can be defined by the simpler formula
\[\mathfrak{R}_{\frakl^X}f=\E_{\frakm}\left[\int_0^1 f(X_s)d\frakl^X_s\right].\]
Its total mass, denoted by $c(\frakm)$, is
\begin{equation} \label{eq:c_nu}
c(\frakm)=\mathfrak{R}_{\frakl^X}\mathbf{1},
\end{equation}
which is a positive constant.\textit{ Since the local time can be defined up to a multiplicative constant, in order to streamline the discussion, we suppose for the remainder of this paper that the local time $\frakl^X$ has been normalized so that $c(\frakm)=1$.} The notations for $\frakl^Y,\tau^Y,\Phi_Y(q),\delta_Y,\gamma_Y,\mu_Y$ are trivial to understand, and we also suppose that $\frakl^Y$ has been normalized to make $c(m)=1$.

Moreover, by Fitzsimmons and Getoor \cite[Proposition (A.4)]{FG_excursion_theory}, since $b$ is regular, we have $\widehat{\Prob}_b[T^{\hatX}_b=0]=1$, where $T^{\hatX}_b$ is hitting time of $\hatX$ to $b$. Let $\hatX^{\dag}=(\hatX_t)_{t<T^{\hatX}_b}$ denote the process $\hatX$ killed at $b$, and $\hatP^{\dag}$ and $\hatU^{\dag}_q$ for its semigroup and resolvent. In addtion, for $x\in E$, we let
\begin{align*}
\varphi^X_q(x)& =\E_x[e^{-qT^X_b}],\varphi^X(x) =\varphi^X_0(x)=\Prob_x[T^X_b<\infty],\hatphi_q(x) = \E_x[e^{-qT^{\hatX}_b}], \hatphi(x)=\hatphi_0(x).
%\varphi^Y_q(x)& =\E_x[e^{-qT^Y_0}],\varphi^Y(x) =\varphi^Y_0(x)=\Prob_x[T^Y_0<\infty],\hatrho_q(x) = \E_x[e^{-qT^{\widehat{Y}}_0}], \hatrho(x)=\hatrho_0(x).
\end{align*}
It is well-known that strong Markov property implies the following relation, for any $x\in E$ and $f\in \Bb(E)\cup \rmL^2(\frakm)$,
\begin{equation}\label{eq:strong_Markov}
U_qf(x) = U_q^{\dag}f(x)+\varphi^X_q(x)U_qf(b).
\end{equation}
On the other hand, although the dual process $\hatX$ is moderate Markov, by \cite[Corollary (A.11)]{FG_excursion_theory}, we have for all $f\in \Bb^{+}(E)$,
\begin{equation}\label{eq:moderate_Markov}
\hatU_qf(x) = \hatU_q^{\dag}f(x)+\hatphi_q(x)\hatU_qf(b).
\end{equation}
Similarly there exists a moderate Markov dual process $\widehat{Y}$ associated with $Y$ and $m$, whose semigroup and resolvent are denoted by $\hatQ$ and $\hatV_q$ respectively. The killed process is denoted by $\hatY^{\dag}$ and its semigroup and resolvent are denoted by $\hatQ^{\dag}$ and $\hatV^{\dag}_q$, and the notations $\varphi^Y_q,\varphi^Y,\hatrho_q,\hatrho$ are self-explanatory.

\section{Statements of main results}
In this section, we will state the main theorem and some of its corollaries. We start by defining a few notations. For two sets $A$ and $B$, we write $A\subsetdense B$ if $A\subseteq B$ and $\overline{A} = B$, where $\overline{A}$ is the closure of $A$. Moreover, for some operator $\Lambda$, we denote  $\D_{\Lambda}$ to be its domain, $Ran(\Lambda)$ its range, and we define the following class of operators
\begin{equation}
\mathcal{C}(m,\frakm)=\{\Lambda:\D_{\Lambda}\subsetdense \rmL^2(m)\rightarrow Ran(\Lambda)\subsetdense\rmL^2(\frakm) \textnormal{ linear, injective and closed.}\}.
\end{equation}
Note that if $\Lambda \in \mathcal{C}(m,\frakm)$, then $\hatLambda \in \mathcal{C}(\frakm,m)$ where $\hatLambda$ is the $\rmL^2$-adjoint of $\Lambda$, i.e.~for any $f\in \D_{\Lambda},g\in \D_{\hatLambda}$, we have $\left\langle \Lambda f,g\right\rangle_{\frakm} = \left\langle  f,\hatLambda g\right\rangle_{m}$, where $\left\langle \cdot, \cdot\right\rangle_{\frakm}$ (resp.~$\left\langle \cdot, \cdot\right\rangle_{m}$) denotes the standard inner product in $\rmL^2(\frakm)$ (resp.~$\rmL^2(m)$). In addition, we say $\Lambda$ is \textit{mass preserving} if $\Lambda \mathbf{1}_E\equiv \mathbf{1}_E$ where $\mathbf{1}_E(x)=1$ for all $x\in E$, and it is assumed that $\mathbf{1}_E$ is in the (possibly) extended domain of $\Lambda$. Then we have the following results.
\subsection{Intertwining relations and inverse local time}
The main result of this section is stated in the following Theorem.
\begin{theorem}\label{thm}
Let $\Lambda\in \mathcal{C}(m,\frakm)$, with both $\Lambda$ and $\hatLambda$ being mass preserving. Consider the following claims.
\begin{enumerate}
\item \label{it:intertwin_killed} $P^{\dag}_t \Lambda f= \Lambda Q^{\dag}_tf$ for all $f\in \D_{\Lambda} \cup \{\mathbf{1}_E\}$.
\item \label{it:intertwin_refl} $P_t \Lambda f= \Lambda Q_tf$ for all $f\in\D_{\Lambda}\cup \{\mathbf{1}_E\}$.
\item \label{it:phi_rho} For any $q> 0$, we have $\varphi^X_q(x)=\Lambda \varphi^Y_q(x)$ $\frakm$-almost everywhere (a.e.~for short) on $E$, and $\hatrho_q(x)=\Lambda \hatphi_q(x)$ $m$-a.e.~on $E$.
\item \label{it:inve_local_time} $\Phi_X(q)=\Phi_Y(q) $ for each $q>0$.
\end{enumerate}
Then, we have \[\eqref{it:intertwin_killed} \Rightarrow \eqref{it:phi_rho} \Rightarrow \eqref{it:inve_local_time}    \textrm{ and } \eqref{it:intertwin_killed} \Rightarrow \eqref{it:intertwin_refl} .\]
If in addition, writing $\mathbf{1}_{\{b\}}$ the indicator function at $\{b\}$, we have
\begin{equation} \label{cond:2imply1}
\begin{aligned}
\Lambda \mathbf{1}_{\{b\}}(x) &= \mathbf{1}_{\{b\}}(x), \quad  \hatLambda \mathbf{1}_{\{b\}}(x) = \mathbf{1}_{\{b\}}(x) \textnormal{ for any $x\in E$, and} \\
\Lambda Q_tf(b) &= Q_tf(b), \quad \hatLambda \hatP_tg(b)=\hatP_tg(b) \textnormal{ for all $f\in \D_{\Lambda}\cup\{\mathbf{1}_E\}, g\in \D_{\hatLambda}\cup\{\mathbf{1}_E\}$,}
\end{aligned}
\end{equation}
then
\[\eqref{it:intertwin_refl} \Rightarrow \eqref{it:phi_rho} \textrm{ and } \eqref{it:intertwin_killed} \Leftrightarrow \eqref{it:intertwin_refl}.\]
\end{theorem}
\begin{remark}
\begin{enumerate}
\item Note that $\Lambda$ can be defined up to a multiplicative constant $c$, hence the mass preserving condition (resp.~condition \eqref{cond:2imply1}) can be stated in a slightly more general way as, there exists a constant $c\neq 0$ such that $c\Lambda$ is mass preserving  (resp.~satisfies \eqref{cond:2imply1}).
\item If $m$ is of finite mass on $E$, then clearly $\mathbf{1}_E\in \rmL^2(m)$. Otherwise, we understand the conditions \eqref{it:intertwin_killed} and \eqref{it:intertwin_refl} as $Q_t$ and $P_t$ acting as a Markov operator on $\Bb(E)$. For sake of simplicity, we keep the same notations as the $\rmL^2$-semigroups.
\end{enumerate}
\end{remark}

\begin{corollary}\label{cor:robin}
Under assumption \eqref{it:intertwin_killed} or equivalently, \eqref{it:intertwin_refl} together with the additional condition \eqref{cond:2imply1} for $\Lambda$, then $\Lambda$ also intertwins two generators with Robin boundary condition at $b$.
\end{corollary}
Here we address that as opposed to the setting in \cite{Arendt_2014}, where there are no similarity transforms between two Laplacians acting on two isospectral domains with Robin boundary condition, our situation is different in two aspects. First, the two generators are acting on the same space and both have the same boundary at 0. Second, the intertwining operator $\Lambda$ that we consider in this paper is not a similarity transform as in \cite{Arendt_2014}. Therefore, we see that under a different setting, there indeed exists an intertwining relation between two Robin type generators.

\subsection{Excursion theory}
We now provide a further probabilistic explanation for the intertwining relation by means of excursion theory. We first recall from Maisonneuve \cite{Maisonneuve_1975} that, for the excursions of $X$ from the regular point $b$, we can associate an exit system $(\bfP,\frakl^X)$ where $\bfP$ is the so-called (Maisonneuve) excursion measure. Moreover, let us define the collection of $\sigma$-finite measures $(\bfP_t)_{t> 0}$ by
\[\bfP_t(f)=\bfP[f(X_t),t<T_b],\]
for any $f\in \Bb^+(E)$. Then $(\bfP_t)_{t>0}$ is an entrance law of semigroup $P^{\dag}$, in other words, $\bfP_{s+t}=\bfP_s P^{\dag}_t$ for any $s>0,t\geq 0$. Furthermore, for any $q>0$, we define $\bfU_q (f)=\int_0^{\infty}e^{-qt}\bfP_t(f)dt$. Similarly, let $\bfQ$ denote the Maisonneuve excursion measure for process $Y$, $(\bfQ_t)_{t> 0}$ be the associated entrance law, and $\bfV_q (f)=\int_0^{\infty}e^{-qt}\bfQ_t(f)dt$. We use $l_X(a)$ (resp.~$l_Y(a)$) to denote the length of the first excursion interval with length $l>a$ for the process $X$ (resp.~$Y$). In addition, we let $M_X$ (resp.~$M_Y$) denote the closure in $[0,\infty)$ of the visiting set $\{t\geq 0; X_t=b\}$ (resp.~$\{t\geq 0; Y_t=b\}$), and $\zeta_X=\sup M_X$ (resp.~$\zeta_Y=\sup M_Y$) be the last exit time of $X$ (resp.~$Y$) from $b$. Then we have the following corollary.
\begin{corollary}\label{prop:excur}
Under the assumption in Theorem \ref{thm} \eqref{it:intertwin_killed}, the following statements hold.
\begin{enumerate}[(a)]
\item For any $A\in \mathcal{B}(\R^+)$ a Borel set, we have $\bfP(T_b^X\in A)= \bfQ(T_b^Y\in A)$.
\item For every $a\in \R_+$, $l_X(a)$ and $l_Y(a)$ have the same distribution.
\item For every $x>0$, $\zeta_X$ and $\zeta_Y$ have the same distribution under $\Prob_x$.
\end{enumerate}
\end{corollary}

\subsection{Krein's spectral theory of strings}\label{sec:Krein_defn}
We first recall that the Laplace exponent of the inverse local time is an essential object in Krein's spectral theory of strings, for which we will provide a brief review of the known results herein, and we refer to \cite{KS_spectral_measure, Krein} for an excellent account. For sake of simplicity, here we take $b=0$ as the regular boundary but note that the choice of 0 is indeed arbitrary. Suppose $Y$ is the Markov process corresponding to the generalized second order differential operator $\bfG=\frac{d}{d\mathtt{m}} \frac{d}{dx}$ with boundary condition $f^{-}(0)=\lim_{x\downarrow 0}\frac{f(0)-f(-x)}{x}=0$, where $\mathtt{m}$ is a string, that is a right-continuous and non-decreasing function defined on $[0,l)\rightarrow [0,\infty)$ for some $0<l=l(\mathtt{m})\leq \infty$ with $\mathtt{m}(0)=0$. Then $Y$ is called a quasi-diffusion (also called generalized diffusion or gap diffusion) with 0 being a regular boundary. In this case, it is known that $\Phi_Y$ is a Pick function, that is, a holomorphic function that preserves the upper half-plane, i.e.~$\Im(\Phi_Y(z))\geq 0$ for all $\Im(z)>0$. Moreover, recalling the L\'evy-Khintchin representation of $\Phi_Y$ as given in \eqref{eq:Levy_Khin_Phi}, then the L\'evy measure $\mu_Y$ admits a density $u_Y$ which is completely monotone, with
\begin{equation} \label{eq:mu_repre}
u_Y(r)=\int_0^{\infty}e^{-rq}\nu_Y(dq),
\end{equation}
for some $\nu_Y$ a measure satisfying $\int_0^{\infty}\frac{\nu_Y(dq)}{1+q}<\infty$, and $\delta_{Y}=\nu_Y(\{0\})$.

Indeed, let $\mathfrak{M}$ and $\mathfrak{P}$  denote the spaces of strings and Pick functions, respectively, then Krein's theory shows that there exists a bijection between $\mathfrak{M}$ and $\mathfrak{P}$, in the sense that for any Pick function $\Phi \in \mathfrak{P}$, there exists a quasi-diffusion $Y$ with generator $\frac{d}{d\mathtt{m}} \frac{d}{dx}$ for some $\mathtt{m}\in \mathfrak{M}$, such that $\Phi$ is the Laplace exponent of the inverse local time of $Y$. The converse also holds. Moreover, recalling that $Q_t^{\dag}$ is the semigroup of $Y$ killed at hitting 0, and let $\bfG^{\dag}$ denote its infinitesimal generator, defined as
\[\bfG^{\dag} f = \lim_{t\rightarrow 0}\frac{Q_t^{\dag}f-f}{t}\]
for $f$ in domain $\D(\bfG^{\dag})=\{f\in \Lg;\: \bfG^{\dag}f\in\Lg\}$. We also recall that a family of orthogonal projection operators $\ttE=(\ttE_q)_{q\in (-\infty,\infty)}$ on $\Lg$ is called a resolution of identity if for all $f \in \Lg$,
\begin{enumerate}
\item $\lim_{q\uparrow r}\ttE_q f = \ttE_r f$, i.e.~$\ttE_q$ is strongly left continuous for all $q\in (-\infty,\infty)$.
\item $\lim_{q\downarrow -\infty}\ttE_q f=0, \lim_{q\uparrow \infty}\ttE_qf=f$.
\item $\ttE_q \ttE_{r}f = \ttE_{\min(q,r)} f$.
\end{enumerate}
Note that since $\bfG^{\dag}$ is a self-adjoint operator, it generates a unique resolution of identity $\ttE^Y=(\ttE^Y_q)_{q\in (-\infty,\infty)}$, which can be represented by
\begin{equation}
\ttE^Y_q = \mathbf{1}_{(-\infty,q]}(\bfG^{\dag}).
\end{equation}
Finally, let $\sigma(\bfG^{\dag})$ represent the spectrum of $\bfG^{\dag}$,
%we may associate a projection valued measure $\mathcal{E}$ defined by, for any $A\in \mathcal{B}$ a Borel set,
%\[\mathcal{E}(A)=\mathbf{1}_A(\bfG^{\dag}),\]
then $Y$ (or its corresponding semigroup $Q$) satisfies \textit{the Krein's property}, which is defined as follows.
\begin{enumerate}
\item For any $f\in \rmL^2(m)$, $Q^{\dag}_tf$ admits the spectral expansion in $\rmL^2(m)$
\begin{equation} \label{eq:spec_Qdag}
Q^{\dag}_t f(x)=\int_{\sigma(\bfG^{\dag})}e^{-qt}d\ttE^Y_qf.
\end{equation}
\item For any $f,g \in \rmL^2(m)$, the signed measure $\left\langle d\ttE^Y_q f,g \right\rangle_{m}$ is absolutely continuous with respect to $\nu_Y(dq)$, the spectral measure of the Pick function $\Phi_Y$ as shown in \eqref{eq:Levy_Khin_Phi} and \eqref{eq:mu_repre},  and the Radon-Nikodym derivative between these two measures is given by
\begin{equation} \label{eq:defn_E}
\frac{\left\langle d\ttE^Y_q f,g \right\rangle_{m} }{\nu_Y(dq)}= (f,h_q)_{m}(g,h_q)_{m}
\end{equation}
for some function $h_q$.
\end{enumerate}
%from the spectral theorem of self-adjoint operators, see e.g.~\cite[Section 12]{Rudin_functional}, we have that
%where

%Furthermore, Krein's theory tells the beautiful fact that  We call this property \textit{the Krein's property}.

During the last decades, there have been a lot of nice developments of Krein's theory of strings, see e.g.~Kotani \cite{Kotani_1975} for a generalization of Krein's theory into the case of singular boundaries. However, these works are still in the framework of quasi-diffusion or differential operator. In what follows, we propose an extension of Krein's theory to general Markov semigroups. Since these linear operators are in general non-self-adjoint operators (neither normal), meaning that there is no spectral theorem available, we need to introduce this weaker notion of resolution of identity. First, fix some interval $[\alpha,\beta]$, $-\infty \leq \alpha < \beta \leq \infty$, we follow \cite{Burnap_Zweifel_1986} to define a \textit{non-self-adjoint resolution of identity} as a family of measure-valued operators $\ttE=(\ttE_q)_{q\in [\alpha,\beta]}: \D(\ttE)\rightarrow\Lnu$ which satisfies the following.
\begin{enumerate}[(i)]
\item $\D(\ttE)\subsetdense\Lnu$ and $\ttE_q\D(\ttE) \subseteq \D(\ttE)$ for all $q\in [\alpha,\beta]$.
\item $\ttE_{\alpha} f=0,\ttE_{\beta}f=f$ for all $f\in \D(\ttE)$.
\item $\ttE_q \ttE_{r}f = \ttE_{\min(q,r)} f$ for all $q,r\in [\alpha,\beta],f\in\D(\ttE)$.
\end{enumerate}
\begin{definition}
Suppose that $\{0\}$ is a regular point for $X$, then we say $X$ (or its corresponding semigroup $P$) satisfies the \textit{weak-Krein property} if  the following conditions hold.
\begin{enumerate}[(i)]
\item The L\'evy measure $\mu_X$ of $\Phi_X$ (the Laplace exponent of the inverse local time at 0) has a completely monotone density, which can be represented in the form \eqref{eq:mu_repre} for some measure $\nu_X$.
\item There exists a Borel set $C$ and  $\D(\ttE^X) \subsetdense \rmL^2(\frakm)$ such that on $\D(\ttE^X)$,
\begin{equation}\label{eq:spec_Pdag}
P^{\dag}_t = \int_{C}e^{-qt}d\ttE^X_q
\end{equation}
for any $t>0$, where $\ttE^X=(\ttE^X_q)_{q\in [\inf C,\sup C]}$ is a non-self-adjoint resolution of identity on $\D(\ttE^X)$.
\item $\left\langle d\ttE^X_qf,g\right\rangle_{\frakm}$ is absolutely continuous with respect to $\nu_X$ for any $f\in \D(\ttE^X),g\in \Lnu$.
\end{enumerate}
\end{definition}
Note that the weak-Krein property only requires the spectral expansion \eqref{eq:spec_Pdag} to hold on a dense subset of $\rmL^2(\frakm)$, which is distinguished from the Krein property for quasi-diffusions, where this expansion holds on the entire Hilbert space. Then we have the following corollary.
\begin{corollary} \label{cor:Krein}
Suppose that Theorem~\ref{thm}\eqref{it:intertwin_killed} holds, with $Q$ being the semigroup of a quasi-diffusion and $\Lambda \in \B(\Lg,\Lnu)$. Further assume that for any $q\in \sigma(\bfG^{\dag})$, $\ttE^Y_q g\in \D_{\Lambda}$ for all $g\in \D_{\Lambda}$, then $P$ has the weak-Krein property, with $C=\sigma(\bfG^{\dag})$.
\end{corollary}

\section{Proof of Theorem~\ref{thm} and its corollaries}
\subsection{Proof of Theorem~\ref{thm}}
We start the proof with the following results, which may of independent interest.
\begin{lemma}\label{lemma}
Assume that  \eqref{it:intertwin_killed} (resp.~\eqref{it:intertwin_refl}) holds, then for any $f\in \D_{\Lambda}$ and $q>0$, we have
%$\Lambda f\in \D(\bfL^{\dag})$ (resp.~$\D(\bfL)$), and
%\begin{eqnarray}
%\bfL^{\dag} \Lambda f &=& \Lambda \bfG^{\dag} f,  \label{eq:intertwin_Ldag} \\
%(\textnormal{resp.~}\bfL \Lambda f &=& \Lambda \bfG f,  \label{eq:intertwin_L})
%\end{eqnarray}
\begin{eqnarray}
U^{\dag}_q\Lambda f &=& \Lambda V^{\dag}_qf.  \label{eq:intertwin_Udag}\\
(\textnormal{resp.~}U_q\Lambda f &=& \Lambda V_qf. \label{eq:intertwin_U} )
\end{eqnarray}
\end{lemma}
\begin{proof}
First, assuming that \eqref{it:intertwin_refl} holds and let us define for any $n>0$, $U_q^nf=\int_0^{n}e^{-qt}P_tfdt$ and $V_q^nf=\int_0^n e^{-qt}Q_tfdt$, then by the intertwining relation, we have, for $f\in \D_{\Lambda}$,
\[U_q^n \Lambda f=\int_0^{n}e^{-qt}P_t\Lambda fdt = \int_0^{n}e^{-qt}\Lambda Q_t fdt= \Lambda \int_0^n e^{-qt}Q_tfdt=\Lambda V_q^nf.\]
However, note that $\lim_{n\rightarrow\infty }V_q^n f= V_qf$ in $\rmL^2(m)$, and $\lim_{n\rightarrow\infty} \Lambda V_q^nf =\lim_{n\rightarrow\infty} U_q^n\Lambda f = U_q\Lambda f$ in $\rmL^2(\frakm)$, then by the closeness property of $\Lambda$, we have
\[\Lambda V_qf = U_q\Lambda f.\]
Similar arguments hold under assumption \eqref{it:intertwin_killed} and this completes the proof.
\end{proof}

\begin{lemma} \label{lemma:phi_in_L2}
For each $q>0$, we have $\varphi_q^X,\hatphi_q\in \rmL^2(\frakm)$ and $\varphi_q^Y,\hatrho_q\in\rmL^2(m)$.
\end{lemma}
\begin{proof}
First, according to Fitzsimmons and Getoor \cite[Theorem (3.6)(ii)]{FG_excursion_theory}, we can write
\[(\mathbf{1}_E,\hatphi_q)_{\frakm}=(\delta_X+q(\varphi^X_q,\hatphi)_{\frakm})U_q \mathbf{1}_E(b).\]
Now since $qU_q \mathbf{1}_E(b)\leq 1$ and $\delta_X+q(\varphi^X_q,\hatphi)_{\frakm}=\Phi_X(q)<\infty$, we see that $(\mathbf{1}_E,\hatphi_q)_{\frakm}<\infty$ for each $q>0$, i.e.~$\hatphi_q \in \rmL^1(\frakm)$ since it is non-negative. Moreover, since $\hatphi_q(x)\leq 1$ for all $x$, we have
\[\int_0^{\infty} \left(\hatphi_q(x)\right)^2 \frakm(dx)\leq \int_0^{\infty} \hatphi_q(x) \frakm(dx) = (\mathbf{1}_E,\hatphi_q)_{\frakm}<\infty.\]
Therefore $\hatphi_q\in \rmL^2(\frakm)$. Similarly, we have
\[(\mathbf{1}_E,\varphi^X_q)_{\frakm}=(\delta_X+q(\hatphi_q,\varphi^X)_{\frakm})\hatU_q \mathbf{1}_E(b).\]
By \cite[Proposition (3.9)]{FG_excursion_theory}, $\delta_X+q(\hatphi_q,\varphi^X)_{\frakm}=\delta_X+q(\varphi^X_q,\hatphi)_{\frakm}<\infty$, while on the other hand $q\hatU_q \mathbf{1}_E(b)\leq 1$, hence $\varphi^X_q \in \rmL^1(\frakm)$ and also in $\rmL^2(\frakm)$ since it is bounded by 1. The same arguments apply for the proof for $\varphi^Y_q$ and $\hatrho_q$, and this completes the proof of this Lemma.
\end{proof}
\subsubsection{Proof of $\eqref{it:intertwin_killed} \Rightarrow \eqref{it:phi_rho}$} Note that for any $x\in E_{\dag}$ where we denote $E_b=E\backslash \{b\}$, we have $\Prob_x(T^X_b=0)=0$, hence since $X$ has an a.s.~infinite lifetime, we can rewrite $\varphi^X_q(x)$ using integration by parts, which yields
\begin{eqnarray*}
\varphi^X_q(x)&=& \int_0^{\infty}e^{-qt}\Prob_x(T^X_b\in dt) = \int_0^{\infty}qe^{-qt}\Prob_x(T^X_b\leq t)dt = 1-\int_0^{\infty}qe^{-qt}P^{\dag}_t \mathbf{1}_E(x)dt\\ &=&1-qU^{\dag}_q\mathbf{1}_E(x),
\end{eqnarray*}
where we used the fact that $P^{\dag}_t \mathbf{1}_E(x)=\Prob_x(T^X_b>t)$. On the other hand, since $b$ is regular for itself, we have $\varphi_q^X(b)=1$. Combining with the fact that $U_q^{\dag} \mathbf{1}_E(b)=0$, we see that for all $x\in E$,
\begin{equation}\label{eq:Udag_varphi}
\varphi^X_q(x)=(\mathbf{1}_E-qU^{\dag}_q\mathbf{1}_E)(x).
\end{equation}
Similarly, we have $\varphi^Y_q(x)=(\mathbf{1}_E-qV^{\dag}_q \mathbf{1}_E)(x)$. Furthermore, by recalling that $\Lambda \mathbf{1}_E=\mathbf{1}_E$ and applying Lemma~\ref{lemma}, we have
\begin{equation}
U^{\dag}_q \mathbf{1}_E(x)=U^{\dag}_q\Lambda \mathbf{1}_E(x)=\Lambda V^{\dag}_q \mathbf{1}_E(x).
\end{equation}
Combining the above results, we get that for any $q>0$ and $x\in E$,
\begin{eqnarray*}
\varphi^X_q(x)&=& (\mathbf{1}_E-qU^{\dag}_q\mathbf{1}_E)(x) = \Lambda \left(\mathbf{1}_E - qV^{\dag}_q \mathbf{1}_E\right)(x) = \Lambda \varphi^Y_q(x).
\end{eqnarray*}
Since we have shown $\varphi^Y_q\in \rmL^2(m)$, we also see that $\varphi^Y_q \in \D_{\Lambda}$. Next, by \eqref{it:intertwin_killed}, we deduce easily the series of identities that for any $f\in \D_{\Lambda}, g\in \D_{\hatLambda}$,
\begin{equation}\label{eq:dual_int}
\left\langle  f,\hatLambda \hatP^{\dag}_t g\right\rangle_{m}=\left\langle \Lambda f,\hatP_t^{\dag}g\right\rangle_{\frakm}=\left\langle P^{\dag}_t\Lambda f,g\right\rangle_{\frakm}=\left\langle \Lambda Q^{\dag}_t f,g\right\rangle_{\frakm}=\left\langle Q^{\dag}_t f,\hatLambda g\right\rangle_{\frakm}=\left\langle f,\hatQ^{\dag}_t \hatLambda g\right\rangle_{\frakm},
\end{equation}
which means that $\hatQ^{\dag}_t \hatLambda g- \hatLambda \hatP^{\dag}_t g \in \D_{\hatLambda}^{\perp} = \{0\}$ since $\overline{\D}_{\hatLambda}=\rmL^2(\frakm)$. Therefore, $\hatP^{\dag}$ and $\hatQ^{\dag}$ have the intertwining relation on $\D_{\hatLambda}$,
\[\hatLambda \hatP_t^{\dag} = \hatQ_t^{\dag} \hatLambda.\]
By \cite[Proposition (A.6)]{FG_excursion_theory}, we have $\widehat{\Prob}_y(T^{\hatY}_b=0)=0$ for all $y\in E_b\backslash S$ where $S$ is an $m$-semipolar set, which $m$ does not charge. On the other hand, since we are assuming that $\hatLambda$ is also mass preserving, we can use the same arguments as above to prove that $\hatLambda \hatphi_q(x)=\hatrho_q(x)$ for all $q>0$ and $x\in E_b\backslash S$. This completes the proof.

\subsubsection{Proof of $\eqref{it:phi_rho}\Rightarrow \eqref{it:inve_local_time}$}
Recall from  \cite[Theorem 3.6]{FG_excursion_theory} that under the normalization $c(\frakm)=1$, the Laplace exponent of the inverse local time can be written as
\begin{equation} \label{eq:defn_Phi}
\Phi_X(q)=\delta_X+q(\varphi^X_q,\hatphi)_{\frakm},
\end{equation}
where we recall that the notation $(\cdot,\cdot)_{\frakm}$ is given in \eqref{defn_bracket}, which means that $(\varphi^X_q,\hatphi)_{\frakm} < \infty$ for all $q>0$. Similarly, $(\varphi^Y_q,\hatrho)_{m}<0$ for all $q>0$. On the other hand, by Lemma~\ref{lemma:phi_in_L2}, we see that $\varphi^X_q,\hatphi_q\in \rmL^2(\frakm)$ and $\varphi^Y_q,\hatrho_q\in \rmL^2(m)$ for any $q>0$. Hence the assumption \eqref{it:phi_rho} implies that for any $q,r>0$,
\[\left\langle\varphi^X_q,\hatphi_r\right\rangle_{\frakm}=\left\langle\Lambda\varphi^Y_q,\hatphi_r\right\rangle_{\frakm}=\left\langle\varphi^Y_q,\hatLambda\hatphi_r\right\rangle{m}=\left\langle\varphi^Y_q,\hatrho_r\right\rangle_{m}.\]
Next, since plainly $\hatphi_r(x) \uparrow \hatphi(x)$ and $\hatrho_r(x) \uparrow \hatrho(x)$ pointwise as $r\downarrow 0$, we easily deduce by monotone convergence that
\[(\varphi^X_q,\hatphi)_{\frakm} = \lim_{r\downarrow 0} (\varphi^X_q,\hatphi_r)_{\frakm}=\lim_{r\downarrow 0} \left\langle\varphi^X_q,\hatphi_r\right\rangle_{\frakm} = \lim_{r\downarrow 0} \left\langle\varphi^Y_q,\hatrho_r\right\rangle_{m} = \lim_{r\downarrow 0} (\varphi^Y_q,\hatrho_r)_{m}= (\varphi^Y_q,\hatrho)_{m},\]
where we used the fact that $(f,g)_m=\left\langle f,g\right\rangle_m$ for any $f,g\in \rmL^2(m)$. %Similar limiting arguments easily show that for all $q>0$, we have
%\[(\hatphi_q,\Lambda \varphi^Y)_{\frakm}=(\hatLambda\hatphi_q,\varphi^Y)_{m},\]
%and assuming $\Lambda \mathbf{1}_E=\mathbf{1}_E$, we have, by linearity of $\Lambda$, that for any $q>0$,
%\begin{equation}\label{eq:delta_no_limit}
%(\hatphi_q,\Lambda(\mathbf{1}_E-\varphi^Y))_{\frakm}=(\hatLambda\hatphi_q,\mathbf{1}_E-\varphi^Y)_{m}.
%\end{equation}
Moreover, from \cite[Remark (3.21)]{FG_excursion_theory}, the killing term $\delta_X$ can be represented as
\begin{eqnarray*}
\delta_X&=&\lim_{q\rightarrow \infty}(\hatphi_q,\mathbf{1}_E-\varphi^X)_{\frakm}=\lim_{q\rightarrow\infty}(\hatphi_q,\Lambda(\mathbf{1}_E-\varphi^Y))_{\frakm}\\&=&\lim_{q\rightarrow\infty}(\hatLambda\hatphi_q,\mathbf{1}_E-\varphi^Y)_{m}=\lim_{q\rightarrow\infty}(\hatrho_q,\mathbf{1}_E-\varphi^Y)_{m}=\delta_Y.
\end{eqnarray*}
Therefore, combining the above results yields
\begin{equation*}
\Phi_X(q)=\delta_X+q(\varphi^X_q,\hatphi)_{\frakm}= \delta_Y+q(\varphi^Y_q,\hatrho)_{m} =  \Phi_Y(q),
\end{equation*}
where we consider again the normalization $c(\frakm)=c(m)=1$. This finishes the proof of $\eqref{it:phi_rho}\Rightarrow \eqref{it:inve_local_time}$.

\subsubsection{Proof of $\eqref{it:intertwin_killed}\Rightarrow \eqref{it:intertwin_refl}$} By \cite[Theorem 3.6 (ii)]{FG_excursion_theory}, for any $f\in \rmL^2(\frakm)$ and $q>0$, $U_qf(b)$ can be written as
\[U_qf(b)=\frac{( f,\hatphi_q)_{\frakm}}{\Phi_X(q)}=\frac{\left\langle f,\hatphi_q\right\rangle_{\frakm}}{\Phi_X(q)},\]
where the second identity comes from Lemma~\ref{lemma:phi_in_L2}. Since we have proved $\eqref{it:intertwin_killed}\Rightarrow \eqref{it:inve_local_time}$ (resp.~$\eqref{it:intertwin_killed}\Rightarrow \eqref{it:phi_rho}$), which means that $\Phi_X=\Phi_Y$ (resp.~$\hatLambda\hatphi_q=\hatrho_q$ $m$-a.e.), we deduce that, for $f\in \D_{\Lambda}$,
\begin{equation}\label{eq:UqVqf0}
U_q\Lambda f(b)= \frac{\left\langle\Lambda f,\hatphi_q\right\rangle_{\frakm}}{\Phi_X(q)}= \frac{\left\langle f,\hatLambda\hatphi_q\right\rangle_{m}}{\Phi_Y(q)}=\frac{\left\langle f,\hatrho_q\right\rangle_{m}}{\Phi_Y(q)}= V_q f(b).
\end{equation}
Furthermore, by \eqref{it:intertwin_killed}, we have $U^{\dag}_q\Lambda f = \Lambda V^{\dag}_q f$, hence the strong Markov property \eqref{eq:strong_Markov} yields that for any $x\in E_b$,
\begin{eqnarray}
U_q\Lambda f=U^{\dag}_q\Lambda f+U_q\Lambda f(b)\varphi^X_q = \Lambda \left(V^{\dag}_qf+ V_qf(b)\varphi^Y_q\right)=\Lambda  V_qf,
\end{eqnarray}
which proves that $P_t\Lambda = \Lambda Q_t$ on $\D_{\Lambda}$ and this completes the proof.

\subsubsection{Proof of $\eqref{it:intertwin_refl} \Rightarrow \eqref{it:phi_rho}$} Now let us further assume condition \eqref{cond:2imply1} for $\Lambda$ and $\hatLambda$. We start by recalling from \cite[Theorem 1]{Rogers_1983} that for any $f\in \rmL^2(\frakm)\cup \{\mathbf{1}_E\}$,
\begin{equation}\label{eq:decomp_Uqfb}
U_qf(b)=\frac{\bfU_q(f)+\gamma_Xf(b)}{\delta_X+q\bfU_q(\mathbf{1}_E)+q\gamma_X}.
\end{equation}
To this end, we will split the proof into three cases, depending on the value of $\delta_X$ and $\gamma_X$.

\textbf{Case 1. $\delta_X>0$.} Let us take $f=\mathbf{1}_E$, then under the condition $\Lambda \mathbf{1}_E=\mathbf{1}_E$,  we combine \eqref{eq:strong_Markov} and \eqref{eq:Udag_varphi} to get, for any $x\in E$,
\begin{eqnarray}\label{eq:Uq1}
U_q\Lambda \mathbf{1}_E(x)&=&U_q\mathbf{1}_E(x)=U^{\dag}_q\mathbf{1}_E(x)+\varphi^X_q(x)U_q\mathbf{1}_E(b)=\frac1q-\frac{\varphi^X_q(x)}{q}+\varphi^X_q(x)U_q\mathbf{1}_E(b)\nonumber \\&=&\frac1q+\left(U_q\mathbf{1}_E(b)-\frac1q\right)\varphi^X_q(x).
\end{eqnarray}
Note that $V_q$ satisfies similar identities as \eqref{eq:strong_Markov} and \eqref{eq:Uq1}, hence by linearity of $\Lambda$, we have
\begin{align*}
\Lambda V_q\mathbf{1}_E(x)=\frac1q+\left(V_q\mathbf{1}_E(b)-\frac1q\right)\Lambda\varphi^Y_q(x).
\end{align*}
Since $U_q\Lambda f = \Lambda V_qf$ by Lemma~\ref{lemma}, we have
\begin{equation}\label{eq:U10_qinverse}
\left(U_q\mathbf{1}_E(b)-\frac1q\right)\varphi^X_q(x)=\left(V_q\mathbf{1}_E(b)-\frac1q\right)\Lambda\varphi^Y_q(x).
\end{equation}
Moreover, by taking $f=\mathbf{1}_E$ in \eqref{eq:decomp_Uqfb}, we see that, under the assumption $\delta_X>0$,
\[U_q\mathbf{1}_E(b)-\frac1q = \frac{\bfU_q (\mathbf{1}_E)+\gamma_X}{\delta_X+q\bfU_q(\mathbf{1}_E)+q\gamma_X}-\frac1q=-\frac{q^{-1}\delta_X}{\delta_X+q \bfU_q(\mathbf{1}_E)+q\gamma_X}<0.\]
On the other hand, using the intertwining relation \eqref{it:intertwin_refl} and the assumptions that $\Lambda Q_tf(b)=Q_tf(b)$,$\Lambda \mathbf{1}_E\equiv \mathbf{1}_E$, we have
\[U_q\mathbf{1}_E(b)=U_q\Lambda \mathbf{1}_E(b)=\Lambda V_q\mathbf{1}_E(b)= V_q\mathbf{1}_E(b),\]
which is a strictly less than $\frac1q$ if $\delta_X>0$. Therefore we can easily conclude from \eqref{eq:U10_qinverse} that $\varphi^X_q(x)=\Lambda \varphi^Y_q(x)$. The dual argument $\hatrho_q(x)=\hatLambda\hatphi_q(x)$ on $E_b\backslash S$ is proved similarly using the dual intertwining relation $\hatLambda \hatP_t = \hatQ_t\hatLambda$, which can be shown via similar methods as \eqref{eq:dual_int}, and the Markov property equation \eqref{eq:moderate_Markov} for $\hatU_q$ and $\hatV_q$.

\textbf{Case 2. $\delta_X=0,\gamma_X>0$.} Since $b$ is regular, we have that $U_q^{\dag}\mathbf{1}_{\{b\}}(x) =0$ for any $x\in E$, and therefore
\begin{equation}\label{eq:UqIb}
U_q\mathbf{1}_{\{b\}}(x)=\varphi_q^X(x)U_q\mathbf{1}_{\{b\}}(b) .
\end{equation}
Recalling the condition $\Lambda \mathbf{1}_{\{b\}} \equiv \mathbf{1}_{\{b\}}$, we therefore have
\begin{eqnarray*}
\varphi^X_q(x)U_q\mathbf{1}_{\{b\}}(b) = U_q \mathbf{1}_{\{b\}}(x) = U_q\Lambda \mathbf{1}_{\{b\}}(x) = \Lambda V_q \mathbf{1}_{\{b\}}(x)=V_q\mathbf{1}_{\{b\}}(b)\Lambda\varphi^Y_q(x),
\end{eqnarray*}
where for the last identity we used a similar argument as in \eqref{eq:UqIb} for $V_q$. Moreover, taking $f=\mathbf{1}_{\{b\}}$ in \eqref{eq:decomp_Uqfb} with $\delta_X=0$, we have
\[U_q\mathbf{1}_{\{b\}}(b) = \frac{\bfU_q (\mathbf{1}_{\{b\}})+\gamma_X\mathbf{1}_{\{b\}}(b)}{q\bfU_q(\mathbf{1}_E)+q\gamma_X}=\frac{\gamma_X}{q\bfU_q(\mathbf{1}_E)+q\gamma_X}>0.\]
Moreover, the assumption $\Lambda Q_t(b)=Q_tf(b)$ yields that
\[U_q\mathbf{1}_{\{b\}}(b)=U_q\Lambda\mathbf{1}_{\{b\}}(b)=\Lambda V_q\mathbf{1}_{\{b\}}(b)= V_q\mathbf{1}_{\{b\}}(b) >0,\]
therefore $\varphi_q^X(x)=\Lambda \varphi_q^Y(x)$. We can prove $\hatrho_q(x)=\hatLambda\hatphi_q(x)$ on $E_b\backslash S$ using similar techniques with the dual intertwining relation $\hatLambda \hatP_t = \hatQ_t\hatLambda$ and identity \eqref{eq:moderate_Markov}.

%\textit{Remark:} Note that by \cite{FG_excursion_theory}, $\gamma_X$ can be represented as
%\begin{align*}
%\gamma_X&=\lim_{q\rightarrow \infty}q(U_q\delta_{\{b\}},\mathbf{1}_E)_{\frakm}=\lim_{q\rightarrow \infty}q(U_q\Lambda\delta_{\{b\}},\mathbf{1}_E)_{\frakm}=\lim_{q\rightarrow \infty}q(\Lambda V_q\delta_{\{b\}},\mathbf{1}_E)_{\frakm}\\
%& = \lim_{q\rightarrow \infty}q( V_q\delta_{\{b\}},\hatLambda\mathbf{1}_E)_{m}=\lim_{q\rightarrow \infty}q( V_q\delta_{\{b\}},\mathbf{1}_E)_{m}=\gamma_Y.
%\end{align*}

\textbf{Case 3. $\delta_X=\gamma_X=0$.} Recall that $(\bfP_t)_{t>0}$ is the (Maisonneuve) entrance law of $P^{\dag}$, and define $\tbfQ_t$ be such that $\tbfQ_t (f)=\bfP_t(\Lambda f)$. Our aim is to show that $\tbfQ_t$ is indeed the Maisonneuve entrance law of $Q^{\dag}$. To this end, we define the measure $\tbfV_0$ on $E_b$ be such that
\[\tbfV_0 (f)=\int_0^{\infty}\tbfQ_s(f)ds.\]
Note that $\tbfV_0 (f) = \bfU_0(\Lambda f)$ as by definition, $\bfU_0 (f)=\int_0^{\infty}\bfP_s(f)ds$. Using the fact that $Q^{\dag}$ is the minimal semigroup, i.e.~$Q^{\dag}_tf\leq Q_tf$ for $f\geq 0$, and together with the intertwining relation \eqref{it:intertwin_refl}, we have for all $f\geq 0$,
\begin{equation} \label{eq:tbfV_bfU}
\tbfV_0( Q^{\dag}_tf) \leq \tbfV_0 (Q_tf)= \bfU_0 (\Lambda Q_tf)=\bfU_0( P_t\Lambda f).
\end{equation}
By \cite[Corollary 3.23]{FG_excursion_theory}, we can write $\bfU_0=\hatphi \frakm|_{E_b}$. Moreover, it is well-known that $\hatphi$ is an excessive function of $\hatP$, hence for any $f\in \rmL^2(\frakm)$,
\[\hatphi \frakm P_tf =(\hatphi,P_tf)_{\frakm}=(\hatP_t\hatphi,f)_{\frakm}\leq (\hatphi,f)_{\frakm}.\]
In other words, the measure $\hatphi \frakm$ is an excessive measure for $P$. However, since we are under the case $\gamma_X=0$, which means $\{b\}$ is a null set for $\frakm$, we see from \eqref{eq:tbfV_bfU} that, for $f\geq 0$,
\[\tbfV_0 (Q^{\dag}_tf)\leq \bfU_0 (P_t\Lambda f) = \hatphi \frakm P_t\Lambda f \leq \hatphi \frakm\Lambda f = \bfU_0 (\Lambda f) = \tbfV_0 (f).\]
Moreover, $\tbfV_0 (Q^{\dag}_t f)  \rightarrow 0$ as $t\rightarrow \infty$, so $\tbfV_0$ is a purely excessive measure for $Q^{\dag}$. Hence by a standard result, see e.g.~\cite[Theorem 5.25]{Getoor_1990}, $\tbfV_0$ is the integral of a uniquely determined entrance law, therefore $\tbfQ_t$ is an entrance law of $Q^{\dag}$. Furthermore, let $\tbfV_q = \int_0^{\infty}e^{-qt}\tbfQ_tdt$, then by \cite{Rogers_1983}, the decomposition of resolvents yields
\[V_qf(b)=\Lambda V_qf(b)=U_q\Lambda f(b)=\frac{\bfU_q(\Lambda f)}{q\bfU_q (\mathbf{1}_{E\backslash\{b\}})}=\frac{\tbfV_q (f)}{q\tbfV_q (\mathbf{1}_{E\backslash\{b\}})},\]
where we used the fact that $\Lambda \mathbf{1}_{E\backslash\{b\}}=\Lambda(\mathbf{1}_E-\mathbf{1}_{\{b\}})=\mathbf{1}-\mathbf{1}_{\{b\}}=\mathbf{1}_{E\backslash\{b\}}$. Hence $\tbfQ_t$ is indeed the Maisonneuve entrance law of $Q^{\dag}$ and $\bfV_q\equiv \tbfV_q$. Finally, we use the relation $\bfV_q = \hatrho_q m|_{E_b}$, see \cite[(3.22)]{FG_excursion_theory}, to get that for any $q>0,f\in \rmL^2(m)\cap \Bb^+(E)$,
\[\left\langle\hatrho_q,f\right\rangle_{m}=\bfV_q (f)=\bfU_q(\Lambda f)=\left\langle\hatphi_q,\Lambda f\right\rangle_{\frakm}=\left\langle\hatLambda \hatphi_q,f\right\rangle_{m},\]
which yields $\hatrho_q(x) = \hatLambda\hatphi_q(x)$ $m$-a.e.~for all $q>0$. The dual relation works similarly.
%{\color{red}{If we take the Dirac mass at 0 $\delta_0$ instead of $\delta_{\{b\}}$, it seems to be fine as long as $\frakm(b),m(b)\neq 0$.}}
\subsubsection{Proof of $\eqref{it:intertwin_refl}\Rightarrow \eqref{it:intertwin_killed}$}
Since \eqref{it:intertwin_refl} implies that $U_q\Lambda f=\Lambda  V_qf$, and we further have $U_q\Lambda f(b)=\Lambda V_qf(b)=V_qf(b)$ under the assumption $\Lambda Q_tf(b)=Q_tf(b)$ for all $f\in \D_{\Lambda}$, hence by simply reordering the strong Markov identity \eqref{eq:strong_Markov}, we have
\[U^{\dag}_q\Lambda f(x)=U_q\Lambda f(x)-U_q\Lambda f(b)\varphi^X_q(x)=\Lambda\left(V_qf(x)-V_qf(b)\varphi^Y_q(x)\right)=\Lambda V^{\dag}_qf(x),\]
where the second identity uses the fact that $\eqref{it:intertwin_refl} \Rightarrow\eqref{it:phi_rho}$. This proves the desired argument.

\subsection{Proof of corollaries}
\begin{proof}[Proof of Corollary~\ref{cor:robin}]
First, by Theorem~\ref{thm}, we have $\Phi_X(q)=\Phi_Y(q)$ and therefore,
\[\gamma_Y = \lim_{q\rightarrow \infty} \frac{\Phi_Y(q)}{q}= \lim_{q\rightarrow \infty} \frac{\Phi_X(q)}{q}=\gamma_X.\]
Moreover, recall that for all $f\in \rmL^2(\frakm)\cup \{\mathbf{1}_E\}$, $U_qf(b)$ can be expressed as \eqref{eq:decomp_Uqfb}, where $\gamma_X$ represents the stickiness of $X$ at point $b$, and similar expression holds for $V_qf(b)$. In other words, when $\gamma_X=\gamma_Y=0$, $b$ is a reflecting boundary for both $X$ and $Y$, hence both processes have a Neumann boundary condition at $b$. While when $\gamma_X=\gamma_Y>0$, both $X$ and $Y$ have a Robin boundary condition at $b$ and this completes the proof.
\end{proof}
\begin{remark}
If $\Lambda$ is a bounded operator with  $\D_{\Lambda}=\rmL^2(m)$, we can also prove this result via infinitesimal generators. In particular, let $\bfL$ (resp.~$\bfG$) denote the infinitesimal generator of $P$ (resp.~$Q$) in $\rmL^2(\frakm)$ (resp.~$\Lg$), and  $\D(\bfL)$ (resp.~$\D(\bfG)$) for its domain. Then for any $ f\in \D(\bfG)$, by the definition of infinitesimal generators, we have  $\lim_{t\rightarrow 0}\frac{Q_tf-f}{t}=\bfG f$ in $\rmL^2(m)$. On the other hand, since $\Lambda \in \B(\rmL^2(m),\rmL^2(\frakm))$, we see that for any sequence $t_n\rightarrow 0$ and $n,k \in \mathbb{N}$,
\[\left\|\Lambda \frac{Q_{t_n}f-f}{t_n}-\Lambda \frac{Q_{t_k}f-f}{t_k}\right\|_{\frakm} \leq |||\Lambda||| \left\|\frac{Q_{t_n}f-f}{t_n}-\frac{Q_{t_k}f-f}{t_k}\right\|_{m} \rightarrow 0,\]
which implies that $\left(\Lambda\frac{Q_{t_n}f-f}{t_n}\right)_{n\geq 0}$ is a Cauchy sequence in $\rmL^2(\frakm)$, and hence convergent. Since $\Lambda$ is also a closed operator, we have
\begin{eqnarray}
\Lambda \bfG f=\Lambda \lim_{t\rightarrow 0}\frac{Q_tf-f}{t}=\lim_{t\rightarrow 0}\frac{\Lambda Q_t f-\Lambda f}{t}=\lim_{t\rightarrow 0}\frac{P_t\Lambda f-\Lambda f}{t},
\end{eqnarray}
where the last identity comes from assumption \eqref{it:intertwin_refl}. Moreover, since $\Lambda$ maps $\rmL^2(m)$ to $\rmL^2(\frakm)$, we have $\Lambda \bfG f\in \rmL^2(\frakm)$ and therefore the right-hand side of the above equation converges in $\rmL^2(\frakm)$. Hence we conclude that $\Lambda f\in \D(\bfL)$ and $\bfL \Lambda f= \Lambda \bfG f$ on $\D(\bfG)$. As both $\bfL$ and $\bfG$ have Robin boundary condition at $b$ when $\gamma_X=\gamma_Y>0$, this completes the proof.
\end{remark}

\begin{proof}[Proof of Corollary~\ref{prop:excur}]
%From  \eqref{eq:Intertwin_killed}, it can be easily observed that
%\[\bfQ_{t+s}=\bfP_{t+s}\Lambda]=\bfP_{s}P^{\dag}_t\Lambda = \bfP_s \Lambda Q^{\dag}_t = \bfQ_s Q^{\dag}_t.\]
%Hence $(\bfQ_t)_{t\geq 0}$ is indeed an entrance law for $Q^{\dag}$.
First, we combine the representation of $\Phi_X$ as in \eqref{eq:Levy_Khin_Phi} and the statement in Theorem~\ref{thm} to make the easy observation that
\begin{equation}\label{eq:mu_relation}
\mu_X(dy)= \mu_Y(dy).
\end{equation}
Hence by \cite[Corollary 2.22]{FG_excursion_theory}, we have
\[\bfP(T_b^X\in A)=\mu_X(A)= \mu_Y(A)= \bfQ(T_b^Y\in A).\]
Note that although the normalizing constants $c(\frakm)$ and $c(m)$ are not 1 in \cite{FG_excursion_theory}, this will not bring any issue because the Maisonneuve excursion measure $\bfP$ and $\bfQ$ are defined up to a multiplicative constant, i.e.~if $(\bfP,\frakl^X)$ is an exit system, then so is $(c^{-1}\bfP,c\frakl^X)$ for any $c>0$. To see this in more detail, we can simply replace $\frakl^X$ by $c(\frakm)\frakl^X$ and $\bfP$ by $\bfP/c(\frakm)$, and note that $\mu_X$ is also replaced by $\mu_X/c(\frakm)$. Similar arguments hold for process $\frakl^Y$ and $\bfQ$ as well, which proves the first item. Moreover, denoting $\overline{\mu}_X(c)=\mu_X(c,\infty)$ for any $c>0$, it is easy to see from \eqref{eq:mu_relation} that $\overline{\mu}_X(c) =  \overline{\mu}_Y(c)$ for any $c>0$. Therefore, by Bertoin \cite[Section IV.2 Lemma 1]{Bertoin_Levy}, for any $b\geq a$, we have
\begin{equation} \label{eq:dist_excur_length}
\Prob(l_X(a)>b) = \frac{\overline{\mu}_X(b)}{\overline{\mu}_X(a)}=\frac{\overline{\mu}_Y(b)}{\overline{\mu}_Y(a)}=\Prob(l_Y(a)>b),
\end{equation}
which proves the second item. Finally, for the last item, we simply apply \cite[Proposition 2.17]{FG_excursion_theory} to get, for any $x,q>0$, that
\[\E_x[e^{-q\zeta_X}]=\frac{\delta_X}{\Phi_X(q)}=\frac{ \delta_Y}{ \Phi_Y(q)}=\E_x[e^{-q\zeta_Y}], \]
Hence $\zeta_X$ and $\zeta_Y$ have the same distribution under $\Prob_x$ and this concludes the proof of this Proposition.
\end{proof}

\begin{proof}[Proof of Corollary~\ref{cor:Krein}]
Given the intertwining relation in \eqref{it:intertwin_killed}, by Theorem~\ref{thm}, we see that $\Phi_X = \Phi_Y$. Moreover, assuming that $Y$ is a quasi-diffusion, which means that $\mu_Y$ has an absolutely continuous density $u_Y$ which admits the representation \eqref{eq:mu_repre} for some measure $\nu_Y$, hence so does $\mu_X$ since we can simply take $\nu_X=\nu_Y$. On the other hand, since $Y$ has the Krein's property, $Q^{\dag}_t$ satisfies the expansion given in \eqref{eq:spec_Qdag}, and there exist functions $(h_q)_{q\in \sigma(\bfG^{\dag})}$ such that
\[\left\langle d\ttE^Y_qf,g\right\rangle_{m} = (f,h_q)_{m}(g,h_q)_{m} \nu_Y(dq),\]
for any $f,g\in \Lg$. Now let us define the family of operators $(\ttE^X_q)_{q\in \sigma(\bfG^{\dag})}$ as $\ttE^X_q  = \Lambda \ttE^Y_q \Lambda^{-1}$ on $\D(\ttE^X)=Ran(\Lambda)$. For any $f\in \D(\ttE^X)$, let $g=\Lambda^{-1}f\in \D_{\Lambda}$, and we observe the following.
\begin{enumerate}[(i)]
\item $\D(\ttE^X)=Ran(\Lambda)$ is assumed to be dense in $\Lnu$. Moreover, for any $q\in \sigma(\bfG^{\dag})$, we have $\ttE^Y_q g \in \D_{\Lambda}$ by assumption. Hence
\[\ttE^X_q f = \Lambda \ttE^Y_q \Lambda^{-1}f = \Lambda \ttE^Y_q g \in \D(\ttE^X),\]
i.e.~$\ttE^X_q \D(\ttE^X)\subseteq \D(\ttE^X)$.
\item Using the property of the resolution of identity $\ttE^Y$ and the boundedness of $\Lambda$, we have
\begin{eqnarray*}
\lim_{q\rightarrow\inf \sigma(\bfG^{\dag})}\ttE^X_{q}f &=& \lim_{q\rightarrow\inf \sigma(\bfG^{\dag})}\Lambda \ttE^Y_{q} \Lambda^{-1}f =\lim_{q\rightarrow\inf \sigma(\bfG^{\dag})}\Lambda \ttE^Y_{q} g =0,\\
\lim_{q\rightarrow\sup \sigma(\bfG^{\dag})}\ttE^X_{q}f &=& \lim_{q\rightarrow\sup \sigma(\bfG^{\dag})}\Lambda \ttE^Y_{q} \Lambda^{-1}f = \Lambda \Lambda^{-1}f = f.
\end{eqnarray*}
\item $\ttE^X_q\ttE^X_r f = \Lambda \ttE^Y_q\Lambda^{-1} \Lambda \ttE^Y_r\Lambda^{-1} f = \Lambda \ttE^Y_{\min(q,r)} \Lambda^{-1}f = \ttE^X_{\min(q,r)} f$ for any $q,r\in \sigma(\bfG^{\dag})$.

\end{enumerate}
Hence $\ttE^X$ is a non-self-adjoint resolution of identity. Next, let $(q_k)_{k=0}^n$ be a partition of $[\inf \sigma(\bfG^{\dag}),\sup \sigma(\bfG^{\dag})]$. Then for any $f\in \D(\ttE^X),g\in \rmL^2(\frakm)$, since $\ttE^Y(\Delta_k)=\ttE^Y_{q_k}-\ttE^Y_{q_{k-1}}$ is an orthogonal projection, we have
\begin{eqnarray*}
\sum_{k=1}^n \left| \left\langle [\ttE^X_{q_k}-\ttE^X_{q_{k-1}}] f,g\right\rangle \right|&=& \sum_{k=1}^n \left| \left\langle \ttE^Y(\Delta_k) \Lambda^{-1}f,\hatLambda g\right\rangle \right| \leq \|\hatLambda g\| \sum_{k=1}^n \|\ttE^Y(\Delta_k)\Lambda^{-1}f\|\\
&\leq & \|\hatLambda g\| \left(\sum_{k=1}^n\|\ttE^Y(\Delta_k)\Lambda^{-1}f\|^2\right)^{\frac12} =\|\hatLambda g\| \left(\sum_{k=1}^n\left\langle \ttE^Y(\Delta_k)\Lambda^{-1}f,\Lambda^{-1}f\right\rangle \right)^{\frac12}\\
&=&\|\hatLambda g\| \|\Lambda^{-1}f\| \leq \|\hatLambda g\| |||\Lambda^{-1}||| \|f\|
\end{eqnarray*}
since the series $\sum_{k=1}^n\left\langle \ttE^Y(\Delta_k)\Lambda^{-1}f,\Lambda^{-1}f\right\rangle$ is telescoping. Therefore, we see that $\left\langle \ttE^X_{\cdot} f,g\right\rangle$ is of bounded variation on $[\inf \sigma(\bfG^{\dag}),\sup \sigma(\bfG^{\dag})]$, and by Riesz representation theorem, there exists a unique operator $\tilde{P}^{\dag}_tf = \int_{\sigma(\bfG^{\dag})} e^{-qt}d\ttE^X_q f$ on $\D(\ttE^X)$. Then it is easy to see that for $f\in \D(\ttE^X)$, $g\in \Lnu$,
\begin{align*}
\left\langle \tilde{P}^{\dag}_tf,g\right\rangle_{\frakm}&=\int_0^{\infty}e^{-qt}d\left\langle \ttE^X_q f,g\right\rangle_{\frakm} = \int_0^{\infty}e^{-qt}d\left\langle \Lambda\ttE^Y_q \Lambda^{-1}f,g\right\rangle_{\frakm}=\int_0^{\infty}e^{-qt}d\left\langle \ttE^Y_q \Lambda^{-1}f,\hatLambda g\right\rangle_{m} \\
&= \left\langle Q_t^{\dag}\Lambda^{-1}f,\hatLambda g\right\rangle_{m} = \left\langle P_t^{\dag} \Lambda \Lambda^{-1}f,g\right\rangle_{\frakm}=\left\langle P_t^{\dag} f,g\right\rangle_{\frakm},
\end{align*}
%\[P_t^{\dag}f = P_t^{\dag} \Lambda \Lambda^{-1}f = \Lambda Q_t^{\dag} \Lambda^{-1}f = \Lambda \int_Ce^{-qt}d\ttE^Y_q\Lambda^{-1}f = \int_C e^{-qt} d\ttE^X_q f.\]
which shows that indeed $P_t^{\dag}f = \tilde{P}_t^{\dag}f$ on $\D(\ttE^X)$. Moreover, for any $f\in \D(\ttE^X),g\in \rmL^2(\frakm)$,
\begin{align*}
\left\langle d\ttE^X_qf,g\right\rangle_{\frakm}&=\left\langle \Lambda d\ttE^Y_q \Lambda^{-1}f,g\right\rangle_{\frakm} = \left\langle d\ttE^Y_q \Lambda^{-1}f,\hatLambda g\right\rangle_{m}\\
& = (\Lambda^{-1}f,h_q)_{m}(\Lambda g,h_q)_{m}\nu_Y(dq)=(\Lambda^{-1}f,h_q)_{m}(\Lambda g,h_q)_{m}\nu_X(dq),
\end{align*}
which means that $\left\langle d\ttE^X_qf,g\right\rangle_{\frakm}$ is absolutely continuous with respect to $\nu_X$ and this shows that $X$ (or its semigroup $P$) also satisfies the weak-Krein property.
\end{proof}

\section{Reflected self-similar and generalized Laguerre semigroups}
The aim of this part is two-fold. On the one hand, we illustrate the main results of the previous sections by studying  two important classes of Markov processes, namely the spectrally negative positive self-similar Markov processes that were introduced by Lamperti \cite{lamperti} and their associated generalized Laguerre processes whose definition will be recalled below. We emphasize that these two classes have been studied intensively over the last two decades and appear in many recent studies in applied mathematics, such as random planar maps, fragmentation equation, biology, see e.g.~\cite{Bertoin_fragmentation}, \cite{Bertoin_Random_planar} and \cite{PS_spectral}. On the other hand, we also provide the spectral expansion of both the minimal and reflected semigroups associated to the generalized Laguerre processes. This complements the work of Patie and Savov \cite{PS_spectral} where such analysis has been carried out for the transient with infinite lifetime generalized Laguerre semigroups. From now on, we fix the Lusin space to be $(E,\mathcal{E})=(\R_+,\mathcal{B}(\R_+))$, the space of Borel sets on non-negative real numbers, and we set $b=0$. Next, we denote by  $\Ybar=(\Ybar_t)_{t\geq 0}$ the  squared Bessel process with parameter $-\theta$, with  $\theta\in(0,1)$, and write $\Qbar=(\Qbar_t)_{t\geq 0}$  its corresponding semigroup, i.e.~$\Qbar_tf(x)=\E_x[f(\Ybar_t)],f\in \Co, x,t\geq 0$. It is well known, see e.g.~\cite[Chapter IV.6]{BS_HandbookBM}, that $\Qbar$ is a Feller semigroup, whose infinitesimal generator is given by
\[\overline{\bfG}f(x)=x f^{''}(x)+(1-\theta)f^{'}(x),\quad x>0,\]
for $f\in \D(\overline{\bfG})=\{f\in \Co;\overline{\bfG} f\in \Co,f^{+}(0)=0\}$  where $f^{+}(x)=\lim_{h\downarrow 0}\frac{f(x+h)-f(x)}{s(x+h)-s(x)}$ is the right derivative of $f$ with respect to the scale function $s(x)=\int^xy^{\theta-1}e^ydy$. Note that $\Qbar$ possesses the so-called 1-self-similarity property, i.e.~for all $t,x,c>0$,
\[ \Qbar_t f(cx) = \Qbar_{c^{-1}t}\rmd_cf(x),\]
where $\rmd_cf(x)=f(cx)$.
%Note that when $\theta=\frac12$, $B$ is indeed a reflected Brownian motion and its generator is the Laplacian with Neumann boundary condition at 0.
Moreover, the measure $\mbar(x)dx=x^{-\theta}dx,x>0,$ is the unique excessive measure for $\Qbar$, and therefore $\Qbar$ admits a unique strongly continuous contraction extension on $\Lm$, also denoted by $\Qbar$ when there is no confusion. Furthermore, note that 0 is a regular reflecting boundary for $\Ybar$, hence we let $\Qbar^{\dag}=(\Qbar^{\dag}_t)_{t\geq 0}$ denote the $\Lm$-semigroup of the killed process $(\Ybar,T_0^{\Ybar})$ where $T_0^{\Ybar} = \inf\{t\geq 0; \Ybar_t=0\}$. Now let the process $Y=(Y_t)_{t\geq 0}$ be defined as
\begin{equation}\label{eq:YZ_relation}
Y_t=e^{-t}\Ybar_{e^{t}-1}, \quad t\geq 0,
\end{equation}
which is the (classical) Laguerre process of parameter $-\theta$, also known as the squared radial Ornstein-Uhlenbeck process with parameter $-\theta$. Its semigroup $Q=(Q_t)_{t\geq 0}$, which admits the representation
\begin{equation} \label{eq:QRrelation}
Q_t f=\Qbar_{e^t-1}\rmd_{e^{-t}} \circ f,
\end{equation}
is also a Feller semigroup in $\Co$ with infinitesimal generator given by
\begin{equation*}
\bfG f(x) = x f^{''}(x) + (1-\theta-x)f^{'}(x), \quad x>0,
\end{equation*}
with $\D(\bfG)=\{f\in \Co;\bfG f\in \Co,f^{+}(0)=0\}$. Moreover, $Q$ admits an invariant measure $m(x)dx$ with density given by
\begin{equation}
m(x)=\frac{x^{-\theta}e^{-x}}{\Gamma(1-\theta)},\quad x>0,
\end{equation}
which is the probability density of a Gamma random variable of parameter $1-\theta$, denoted by $G(1-\theta)$. Therefore, $Q$ admits a strongly continuous contraction extension on $\Lg$, also denoted by $Q$ when there is no confusion. It is well-known that $Q$ is self-adjoint in $\Lg$ with a spectral decomposition given in terms of the (classical) Laguerre polynomials, see e.g.~\cite[Section 2.7.3]{Bakry_2014}. We also let $Q^{\dag}=(Q^{\dag}_t)_{t\geq 0}$ be the $\Lg$-semigroup of the killed process $(Y,T_0^Y)$ since $0$ is also a reflecting boundary for $Y$.

We proceed by introducing two classes of Markov processes with jumps  which are  natural generalizations of the processes $\Ybar$ and $Y$ in the sense that they share the 1-self-similarity property of $\Ybar$  and the second class is constructed from the first one by means of the relation \eqref{eq:YZ_relation}. To this end, let $\xi=(\xi_t)_{t\geq 0}$ be a spectrally negative L\'evy process, which is possibly killed at a rate $\kappa \geq 0$, that is, killed at an independent exponential time with parameter $\kappa$. It is then well-known that $\xi$ can be characterized by its Laplace exponent $\psi:\C_+ = \{ z\in \C: \Re(z)\geq 0\} \rightarrow \C$, which is defined, for any $\Re(z)\geq 0$, by
\begin{equation} \label{eq:psi}
\psi(z) = \beta z+\frac{\sigma^2}{2}z^2-\int_{0}^{\infty} (e^{-zy}-1+zy\mathbf{1}_{|y|<1}) \Pi(dy) - \kappa,
\end{equation}
where $ \beta \in\R,\sigma\geq 0, \kappa \geq 0$, and $\Pi$ is a $\sigma$-finite measure satisfying $\int_0^{\infty}(y^2\wedge y)\Pi(dy)<\infty$. Note that the quadruplet $(\beta,\sigma,\Pi,\kappa)$ uniquely determines $\psi$ and therefore uniquely determines $\xi$. Furthermore, let
\begin{equation}
\mathbf{T}(t) = \inf\left\{ s>0;\int_0^s e^{\xi_r}dr > t \right\},
\end{equation}
and for an arbitrary $x>0$, define the process $\mathbf{\overline{X}}=(\mathbf{\overline{X}}_t)_{t\geq 0}$ by
\begin{equation} \label{eq:lamp}
\mathbf{\overline{X}}_t = xe^{\xi_{\mathbf{T}(tx^{-1})}}, \quad t\geq 0,
\end{equation}
where the above quantity is assumed to be 0 when $\mathbf{T}(tx^{-1}) = \infty$. According to Lamperti \cite{lamperti}, $\mathbf{\overline{X}}$ is a 1-self-similar Markov process,  and  its infinitesimal generator takes the form
\begin{eqnarray} \label{eq:igd}
\overline{\bfL} f(x) &=& \sigma^2 xf''(x)+\left(\beta+\sigma^2\right)f'(x)+\int_{0}^{\infty}\left(f(e^{-y}x)-f(x)+yxf'(x)\right) \frac{\Pi(dy)}{x}-\kappa f(x),
\end{eqnarray}
for at least functions $ f \in \D_{\bfL}=\{f_e(\cdot)=f(e^{\cdot})=\mathtt{C}^2([-\infty,\infty])\}$. Next, writing the set $\N=\{\psi \textnormal{ of the form \eqref{eq:psi}} \}$, the Lamperti transformation  \eqref{eq:lamp} enables to define a bijection between the
subspace of negative definite functions $\N$ and the 1-self-similar processes $\overline{\mathbf{X}}$. Moreover, when
\begin{eqnarray}
\psi \in \N_{\uparrow} &=& \{\psi\in\N;\: \beta\geq 0,\kappa=0\}
\end{eqnarray} then $\mathbf{\overline{X}}$ never reaches $0$ and has an a.s.~infinite lifetime. Otherwise,  if  $\psi \in \N\setminus \N_{\uparrow}$, then
0 is an absorbing point, which is reached  continuously if $\kappa=0$ and $\beta<0$ or by a  jump  if $\kappa>0$. In addition, according to Rivero \cite{rivero},  see also Fitzsimmons \cite{Fitzsimmons_06}, for each
\begin{eqnarray}
\psi \in \N_{\checkmark} &=& \left\{\psi\in\N;\:  \exists \: \theta\in(0,1) \textnormal{ such that $\psi(\theta)=0$ and }\int_{x>1} xe^{\theta x}\Pi(dx)<\infty\right\},
\end{eqnarray}
$\mathbf{\overline{X}}$ admits a unique recurrent extension that leaves a.s.~0 continuously, denoted by  $\Xbar=(\Xbar_t)_{t\geq 0}$. Its minimal process $\Xbar^{\dag}=(\Xbar^{\dag}_t)_{t\geq 0}=(\Xbar_t;0\leq t\leq T_0^{\Xbar})$ is equivalent to $\mathbf{\overline{X}}$, and 0 is a regular boundary for $\Xbar$. Let $\Pbar=(\Pbar_t)_{t\geq 0}$ and $\Pbar^{\dag}=(\Pbar^{\dag}_t)_{t\geq 0}$ denote the Feller semigroups of $\Xbar$ and $\Xbar^{\dag}$, respectively, i.e.~$\Pbar_t f(x)=\E_x[f(\Xbar_t)],\Pbar^{\dag}_t f(x)=\E_x[f(\Xbar_t),t<T_0^{\Xbar}],f\in \Co$. We also deduce from \cite[Lemma 3]{rivero} that $\mbar$ is, up to a multiplicative constant, the unique excessive measure for $\Pbar$ and also an excessive measure for $\Pbar^{\dag}$, hence both $\Pbar$ and $\Pbar^{\dag}$ can be uniquely extended to a strongly continuous contraction semigroup on $\Lm$, still using the same notations when there is no confusion.

Moreover, we define the process $X=(X_t)_{t\geq 0}$ by $X_t = e^{-t}\Xbar_{e^t-1}, t\geq 0$, which, by the self-similarity property of $\Xbar$ is a homogeneous Markov process and is called a  \textit{reflected generalized Laguerre process}, with 0 also being a regular boundary.  $X^{\dag}=(X^{\dag}_t)_{t\geq 0}$ stands for its minimal process, that is the one  killed  at the stopping time $T_0^X$. Note that, due to the deterministic and bijective transform between processes $X$ and $\Xbar$, $X$ can also be uniquely characterized by $\psi\in \N_{\checkmark}$. We further let $P=(P_t)_{t\geq 0}$ and $P^{\dag}=(P^{\dag}_t)_{t\geq 0}$ denote the Feller semigroups of $X$ and $X^{\dag}$, respectively. Then we easily get that
\begin{equation} \label{eq:PKrelation}
P_t f=\Pbar_{e^t-1}\rmd_{e^{-t}} \circ f,
\end{equation}
and the infinitesimal generator of $P$ is given, for $f\in \D_{\bfL}$, by
\begin{equation}
\bfL f(x) = \overline{\bfL}f(x)-xf'(x).
\end{equation}
%If $\psi\in \N_{\uparrow}$, $Z$ is a 1-self-similar process with infinite lifetime and $X$ is a so-called generalized Laguerre process, which is introduced and studied in Patie and Savov \cite{PS_spectral}. On the other hand, when $\psi\in \N_{\checkmark}$, by Rivero \cite{rivero}, $Z$ shall be viewed as the unique self-similar recurrent extension of the minimal process $Z^{\dag}$, which is the process killed by continuously hitting 0 when $q=0$, and is killed at an independent exponential rate $q$ when $q>0$. Furthermore, and this will be the situation we mainly consider in the rest of the paper.
We observe that $\Ybar$ and $Y$ are special instances of  $\Xbar$ and $X$ respectively, when $\kappa=0$ and $\Pi \equiv 0$ in \eqref{eq:igd}.
%
%In this section, we will provide two intertwining relations, one between $R$ and $K$, one betwee $Q$ and $P$. Note that the former is equivalent to an intertwining relation with the Bessel generator, which contains the Laplacian as a special case.
%\begin{enumerate}
%\item There exists $\kappa>0$ such that the limit $\lim_{x\rightarrow 0}\frac{\E_x[1-e^{-T_0}]}{x^{\kappa}}$ exists and is strictly positive.
%\item The limit $\lim_{x\rightarrow 0}\frac{U_q^{\dag}f(x)}{x^{\kappa}}$ exists for all $f\in \ttC_K(\R_+)$ and is strictly positive for some functions.
%\item There exists $\theta\in(0,1)$ such that $\psi(\theta)=0$,
%\item $\int_{x>1} xe^{\theta x}\Pi(dx)<\infty$.
%\end{enumerate}
%This set of conditions will be referred to as the \textit{Rivero's conditions} in the rest of this paper. We mention that this class of $P$ nicely complements the generalized Laguerre semigroup studied by \cite{PS_spectral}, and our study also relies on their results stated therein.
%Finally, $P$ admits an invariant measure, which is a probability measure on $\R_+$ whose law is  absolutely continuous with a density denoted by $\frakm$, such that for any $f\in \Co$ and $t\geq 0$,
%\begin{equation}
%\frakm P_tf=\frakm f.
%\end{equation}
Before stating the main result of this section, we need to introduce a few additional objects. First, we recall that the Wiener-Hopf factorization for spectrally negative L\'evy processes, see e.g.~\cite{Kyprianou}, yields that the function $\phi$ defined by $\phi(u)=\frac{\psi(u)}{u-\theta}, u\geq0,$ is a Bernstein function, that is the Laplace exponent of a subordinator $\eta=(\eta_t)_{t\geq 0}$ (i.e.~a non-decreasing L\'evy process), see e.g.~\cite{Schilling_Berstein} for an excellent account of Bernstein functions. Then, for $f\in\Co$ we define the Markov multiplier $\Ip$ by
\begin{equation} \label{eq:defn_Ip}
\quad \Ip f(x)=\E[f(xI_{\phi})]
\end{equation}
where $I_{\phi}=\int_0^{\infty}e^{-\eta_t}dt$ is the so-called exponential functional of $\eta$, see e.g.~\cite{Patie-Savov-Exp} and the references therein for a recent account  on this variable. We are now ready to state the following.
\begin{theorem} \label{thm:gL_main}
For each $\psi \in \N_{\checkmark}$, the following statements hold.
\begin{enumerate}
\item \label{it:frakm} There exists a  positive random variable $\Vp$ whose law is absolutely continuous with a density denoted by  $\frakm$, and it is an invariant measure for the semigroup $P$. Moreover, the law of $\Vp$ is determined by its entire moments
\begin{equation} \label{eq:moment_Vp}
\M_{\Vp}(n+1)=\prod_{k=1}^n\frac{\psi(k)}{k}, \quad n\in \mathbb{N}.
\end{equation}
\item  \label{it:Ip}$\Ip\in \B(\Co)\cap \B(\Lm) \cap \B(\Lg, \Lnu)$ and has a dense range in both $\Lm$ and $\Lnu$. Furthermore, both $\Ip$ and $\hatLambda_{\phi}$ are mass-preserving and satisfy the condition \eqref{cond:2imply1}.
\item \label{it:intertwin_ss} For all $f\in \Lm$ (resp.~$f\in \Lg$), we have
\begin{equation} \label{eq:intertwin_refl_ss}
\Pbar_t\Ip f = \Ip \Qbar_tf \quad (resp.~P_t\Ip f = \Ip Q_tf),
\end{equation}
and consequently,
\begin{equation} \label{eq:intertwin_killed_ss}
\Pbar^{\dag}_t \Ip f= \Ip \Qbar^{\dag}_tf \quad (resp.~P^{\dag}_t \Ip f= \Ip Q^{\dag}_tf).
\end{equation}
\item \label{it:Phi} Under the normalization $c(\mbar)=c(\frakm)=c(m)=1$, we have for any $q>0$,
\begin{equation}
\Phi_{\Ybar}(q)=\Phi_{\Xbar}(q) = \frac{\Gamma(1-\theta)}{\Gamma(\theta)}2^{1-\theta}q^{\theta}, \quad \Phi_X(q)=\Phi_Y(q)=\frac{\theta \Gamma(q+\theta)}{\Gamma(1+\theta)\Gamma(q)}.
\end{equation}
\item \label{it:Krein} $\Xbar$ and $X$ satisfy the weak-Krein property.
\end{enumerate}
\end{theorem}
\begin{remark}\label{remark_gL}
\begin{enumerate}[(i)]
\item The expression of the entire moment of $V_{\psi}$ appears in the work of Barczy and D\"oering \cite[Theorem 1]{Barczy_2013}. Their proof rely on a representation as the solution of stochastic differential equation of some recurrent extensions of Lamperti processes. We shall provide an alternative proof which is in the spirit of the papers of Rivero \cite{rivero} and Fitzsimmons \cite{Fitzsimmons_06} and could be used in a more general context.
\item To prove \eqref{eq:intertwin_refl_ss}, we shall resort to a criteria that was developed in \cite{CPY_exp_stoch}, and the details of this proof can be found in Section~\ref{sec:proof_intertwin}.  Note that a crucial assumption is the conservativeness of the semigroups (i.e.~$\Pbar_t \mathbf{1}=\mathbf{1},P_t \mathbf{1}=\mathbf{1}$), a property that is not fulfilled by $\Pbar^{\dag}$ or $P^{\dag}$. Instead, to prove \eqref{eq:intertwin_killed_ss}, we use our Theorem~\ref{thm}, revealing its usefulness in this context.
\item It is well-known that the local time is defined up to a normalization constant. In this paper, it is considered as an additive functional whose support is $\{0\}$ and with the total mass of its asociated Revuz measure normalized to $c(\mbar)=c(\frakm)=c(m)=1$. However, one can also view the local times of $\Ybar$ and $Y$ as the unique increasing process in the Doob-Meyer decomposition of the semi-martingale $(\Ybar_t^{\theta})_{t\geq 0}$ and $(Y_t^{\theta})_{t\geq 0}$ respectively, see e.g.~\cite[Theorem 3.2]{HY_Inv_Local_time_OU}, which are denoted by $\tilde{\frakl}^{\Ybar}$ and $\tilde{\frakl}^{Y}$. The local times for $\Xbar$ and $X$ can be defined similarly, see Section~\ref{sec:proof_phi} for the proof. Under this definition, the total mass of the Revuz measure is given by
\begin{equation}
\tilde{c}(\frakm)=\frac{\theta W_{\phi}(1+\theta)}{\Gamma(1-\theta)\Gamma(1+\theta)}, \quad \tilde{c}(m) = \frac{\theta}{\Gamma(1-\theta)},
\end{equation}
where $W_{\phi}$ will be defined later in the context. Under this normalization, the corresponding Laplace exponents take the form
\begin{equation}\label{eq:Phi_Y_HY}
\tilde{\Phi}_X(q)=\frac{\Gamma(1-\theta) \Gamma(q+\theta)}{W_{\phi}(1+\theta)\Gamma(q)} ,\quad \tilde{\Phi}_Y(q)=\frac{\Gamma(1-\theta)}{\Gamma(1+\theta)}\frac{\Gamma(q+\theta)}{\Gamma(q)}.
\end{equation}
We will detail this computation in Section~\ref{sec:proof_phi}.
\item The intertwining relation \eqref{eq:intertwin_refl_ss} is also a useful tool for deriving the spectral expansion of $P_tf$ and $P^{\dag}_tf$ in $\rmL^2(\frakm)$ under various conditions. We will provide such expansions in Section~\ref{sec:spectral}.
\end{enumerate}
\end{remark}
The rest of this section is devoted to proving Theorem~\ref{thm:gL_main}.

\subsection{Proof of Theorem~\ref{thm:gL_main}\eqref{it:frakm}, \eqref{it:Ip} and \eqref{it:intertwin_ss}} \label{sec:proof_intertwin}
First, let us prove that the expression of the entire moments of  the variable  $\Xbar_1$ under $\Prob_0$ is given by \eqref{eq:moment_Vp}.  Writing $\psi_{\uparrow}(u)=\psi(u+\theta),u\geq 0,$ we observe that
\begin{equation*}
\psi_{\uparrow}(0)=\psi(\theta)=0, \quad \psi_{\uparrow}(u)>0 \textnormal{ for }u>0, \quad \psi_{\uparrow}^{'}(0+)=\psi^{'}(\theta) >0,
\end{equation*}
hence $\psi_{\uparrow} \in \N_{\uparrow}$ is the Laplace exponent of a spectrally negative L\'evy process $\xi^{\uparrow}$, which drifts to $+\infty$ a.s. and is associated, via the Lamperti mapping, to a 1-self-similar process which can be viewed as the minimal process $X^{\dag}$ conditioned to stay positive. Let $I_{\psi_{\uparrow}}=\int_0^{\infty} e^{-\xi^{\uparrow}_t}dt$ denote the exponential functional of $\xi^{\uparrow}$, which, by \cite[Theorem 1]{BY_exp_functional}, is well-defined, i.e.~$I_{\psi_{\uparrow}}<\infty$ a.s., and has negative moments of all orders, see \cite[Theorem 3]{BY_exp_functional}. We also let $\overline{U}_qf(x) = \int_{0}^{\infty} e^{-qt} \Pbar_tf(x) dt$ denote the resolvent of the self-similar semigroup $\Pbar$. Then combining \cite[Theorem 2]{rivero} and \cite[Equation (4)]{BY_exp_functional}, with $p_z(x)=x^z$, we get, for each $q>0,\Re(z)\geq 0$,
\begin{equation} \label{eq:resolvent_1}
\begin{aligned}
\overline{U}_qp_z(0)&=\frac{1}{\M_{I_{\psi_{\uparrow}}}(\theta)\Gamma(1-\theta) q^{\theta}}\M_{I_{\psi_{\uparrow}}}(-z+\theta)\int_0^{\infty}e^{-qt}t^{z-\theta}dt\\
&= \frac{\Gamma(z-\theta+1)}{\Gamma(1-\theta)} \frac{\M_{I_{\psi_{\uparrow}}}(-z+\theta)}{\M_{I_{\psi_{\uparrow}}}(\theta)}p_{-z-1}(q).
\end{aligned}
\end{equation}
On the other hand, from the definition of the resolvent $\overline{U}_q$ and the 1-self-similarity of $\Pbar$, we have
\begin{equation}\label{eq:resolvent_2}
\overline{U}_qp_z(0) = \int_{0}^{\infty} e^{-qt} \Pbar_tp_z(0) dt=\M_{\Vp}(z+1)\int_0^{\infty}e^{-qt}t^zdt=\M_{\Vp}(z+1)\Gamma(z+1)p_{-z-1}(q).
\end{equation}
Combining equation \eqref{eq:resolvent_1} and \eqref{eq:resolvent_2}, we deduce that
\begin{equation} \label{eq:Vp_and_I_psitheta}
\M_{\Vp}(z+1) = \frac{\Gamma(z-\theta+1)}{\Gamma(1-\theta)\Gamma(z+1)}\frac{\M_{I_{\psi_{\uparrow}}}(-z+\theta)}{\M_{I_{\psi_{\uparrow}}}(\theta)}=\M_{B(1-\theta,\theta)}(z+1)\frac{\M_{I_{\psi_{\uparrow}}}(-z+\theta)}{\M_{I_{\psi_{\uparrow}}}(\theta)},
\end{equation}
where $B(1-\theta,\theta)$ is a random variable following a Beta distribution with parameters $( 1-\theta,\theta)$. By \cite[(2.3)]{PS_Weierstrass}, the Mellin transform of $I_{\psi_{\uparrow}}$ satisfies the functional equation
\begin{equation} \label{eq:FE_for_I_psitheta}
\M_{I_{\psi_{\uparrow}}}(-z+1) = \frac{z}{\psi_{\uparrow}(z)}\M_{I_{\psi_{\uparrow}}}(-z),
\end{equation}
which holds on the domain $\{z\in \mathbb{C}: \psi_{\uparrow}(\Re(z))\leq 0\}$. Combining \eqref{eq:FE_for_I_psitheta} and \eqref{eq:Vp_and_I_psitheta}, we get, for $\Re(z)\geq 0$, that
\begin{equation*}
\frac{\M_{\Vp}(z+1)}{\M_{\Vp}(z)} = \frac{\Gamma(z)}{\Gamma(z+1)}\frac{\Gamma(z-\theta+1)}{\Gamma(z-\theta)}\frac{\M_{I_{\psi_{\uparrow}}}(-z+\theta)}{\M_{I_{\psi_{\uparrow}}}(-z+\theta+1)} = \frac{z-\theta}{z}\frac{\psi_{\uparrow}(z-\theta)}{z-\theta}=\frac{\psi_{\uparrow}(z-\theta)}{z}=\frac{\psi(z)}{z}.
\end{equation*}
Hence \eqref{eq:moment_Vp} can be easily observed from the above relation together with the initial condition $\M_{\Vp}(1)=1$. Next, the estimates
\begin{equation*}
\left|\frac{\frac{\prod_{k=1}^{n+1}\psi(k)}{((n+1)!)^2}}{\frac{\prod_{k=1}^n\psi(k)}{(n!)^2}}\right| = \left|\frac{\psi(n+1)}{(n+1)^2}\right| \rightarrow \left\{\begin{array}{rcl}\frac{\sigma^2}{2} &\mbox{if} &\sigma^2>0 \\
0 & \mbox{if} & \sigma^2=0 \end{array}\right. \quad \textnormal{ as } n\rightarrow \infty,
\end{equation*}
yields that the series
\begin{equation}\label{Laplace_trans_VPsi}
\E[e^{q\Vp}]=\sum_{n=1}^{\infty}\frac{\M_{\Vp}(n+1)}{n!}q^n = \sum_{n=1}^{\infty}\frac{\prod_{k=1}^n\psi(k)}{(n!)^2}q^n
\end{equation}
converges for $|q|<\frac{2}{\sigma^2}$ when $\sigma^2>0$ and converges for $|q|<\infty$ when $\sigma^2=0$. Therefore, we get that $\Vp$ is moment determinate. This completes the proof of Theorem~\ref{thm:gL_main}\eqref{it:frakm}. Now, combining \cite[Theorem 1]{rivero} and \cite[Proposition 2.4]{PP_refined_factorization} combined, we obtain that the law of $\Vp$ is absolute continuous and we denote its density $\frakm$. Then, we write, for any $t,x>0$,
\[t \frakn_t(tx)=\frakm(x),\]
i.e.~changing slightly notation here and below $\int_0^{\infty} f(x)\frakm(x)dx=\frakm f = \frakn_t \rmd_{1/t} f$. Then, combining \eqref{eq:moment_Vp} with the self-similarity property of $\Pbar$ identifies $(\frakn_t(x)dx)_{t\geq 0}$ as a family of entrance laws for $\Pbar$, that is, for any $t,s>0$ and $f\in \Co$, $\frakn_t \Pbar_s f = \frakn_{t+s}f.$
Next, using successively the relation \eqref{eq:PKrelation}, the previous identity with $t=1$ and $s=e^t-1$, and the definition of $\frakn_t$ above, we get that, for any $t>0$,
\[\frakm P_tf = \frakm \Pbar_{e^t-1} \rmd_{e^{-t}} \circ f = \frakn_{e^t} \rmd_{e^{-t}}\circ f = \frakm f.\]
Hence,  $\frakm(x)dx$ is an invariant measure for $P$. Therefore, $P$ can be uniquely extended to a strongly continuous contraction semigroup on $\Lnu$, also denoted by $P$ when there is no confusion.
%Let us recall from previous discussion that $\frakm$ is the time 1 entrance law for semigroup $\Pbar$, and we denote the corresponding random variable by $\Vp$.

Next, we proceed by proving Theorem~\ref{thm:gL_main}\eqref{it:Ip}. The fact that $\Ip\in\B(\Co)$ follows immediately by dominated convergence. For any $f\in \Lm$, we use the Cauchy-Schwarz inequality and a change of variable to deduce that
\[\|\Ip f\|_{\mbar}^2 \leq \E\left[\int_0^{\infty}f^2(xI_{\phi})\mbar(x)dx\right]=\M_{I_{\phi}}(\theta)\int_0^{\infty}f^2(x)\mbar(x)dx = \M_{I_{\phi}}(\theta) \|f\|_{\mbar}^2.\]
Since $\M_{I_{\phi}}(\theta)<\infty$ by \cite[Proposition 6.8]{PS_spectral}, we get that $\Ip \in \B(\Lm)$. In order to prove that the range of $\Ip$ is dense in $\B(\Lm)$, we first define the following function, for $ \Re(z)\in \left(\frac{\theta}{2},\frac{\theta}{2}+1\right),$
\begin{equation}
\M_g(z)=\frac{W_{\phi}(-z+\frac{\theta}{2}+1)\Gamma(z-\frac{\theta}{2})}{\Gamma(-z+\frac{\theta}{2}+1)},
\end{equation}
where $W_{\phi}$ is the unique log-concave solution to the functional equation $W_{\phi}(z+1)=\phi(z)W_{\phi}(z)$ for $\Re(z)\geq 0$, with initial condition $W_{\phi}(1)=1$, see \cite[Theorem 5.1]{PS_spectral} and \cite{Patie-Savov-Exp} for a comprehensive study of this equation. Using the Stirling formula, see e.g.~\cite[(2.1.8)]{PK_Mellin},
\begin{equation} \label{eq:stir}
\left|\Gamma(z)\right| = C|e^{-z}||z^z||z|^{-\frac12}(1 + o(1)), \quad C>0,
\end{equation}
which is valid for large $|z|$ and $| \arg(z)| < \pi$,
as well as the large asymptotic behavior, along the imaginary line $\frac12+ib$, of $W_{\phi}$, see \cite[Theorem 5.1(3)]{PS_spectral}, we have
\begin{equation} \label{eq:asymp_Mg}
\M_g\left(\frac12+ib\right) ={\rm{o}}\left(|b|^{-\theta-u}\right)
\end{equation}
as $|b|\rightarrow\infty$, for any $u> \frac12-\theta$. $\M_g$ being analytical on the strip $\Re(z)\in \left(\frac{\theta}{2},\frac{\theta}{2}+1\right)$,  it is therefore absolutely integrable and decays to zero uniformly along the lines of this strip. Hence one can apply the Mellin inversion theorem which combines with the Cauchy's Theorem, see e.g.~\cite[Lemma 3.1]{Patie_Zhao_17} for details of a similar computation, gives that
\[g(x)=\frac{1}{2\pi i}\int_{\frac12-i\infty}^{\frac12+i\infty} x^{-z}\M_g(z)dz=\sum_{n=0}^{\infty} \frac{(-1)^n W_{\phi}(n+1)}{(n!)^2}x^{n-\frac{\theta}{2}}.\]
On the other hand, again by \eqref{eq:asymp_Mg}, one easily observes that the mapping $b \mapsto \M_g(\frac12+ib) \in \rmL^2(\R)$ and therefore, by the Parseval identity of the Mellin transform, we have $g\in \rmL^2(\R_+)$, which further yields that
\[g^{(\theta)}(x)=x^{\frac{\theta}{2}}g(x)= \sum_{n=0}^{\infty} \frac{(-1)^n W_{\phi}(n+1)}{(n!)^2}x^n \in \Lm.\]
Moreover, we recall from \cite{BY_moments} that the law of $I_{\phi}$ is absolutely continuous, with a density denoted by $\iota$, and is determined by its entire moments
\begin{equation}\label{eq:Mellin_I_phi}
\M_{I_{\phi}}(n+1)=\E[I_{\phi}^n]=\frac{n!}{\prod_{k=1}^n\phi(k)}=\frac{n!}{W_{\phi}(n+1)}, n \in \mathbb{N}.
\end{equation}
Hence, by means of a standard application of Fubini theorem, see e.g.~\cite[Section 1.77]{Titchmarsh39}, one shows that, for any $c,x>0$,
\[ \Ip \rmd_c g^{(\theta)}(x)= \sum_{n=0}^{\infty} \frac{(-1)^n W_{\phi}(n+1)}{(n!)^2}(cx)^n \M_{I_{\phi}}(n+1)=\sum_{n=0}^{\infty} \frac{(-1)^n }{n!}(cx)^n=\rmd_c \e(x),\]
where $\e(x)=e^{-x} \in \Lm$. Since the span of $(\rmd_c \e)_{c>0}$ is dense in $\Lm$, we conclude that $\Ip$ has a dense range in $\Lm$. Next, combining \eqref{eq:moment_Vp} and \eqref{eq:Mellin_I_phi}, we obtain that, for all $n\in \mathbb{N}$,
\begin{equation*}
\M_{V_{\psi}}(n+1)\M_{I_{\phi}}(n+1) =\frac{\prod_{k=1}^n (k-\theta)\phi(k)}{\prod_{k=1}^n\phi(k)}=\frac{\Gamma(n+1-\theta)}{\Gamma(1-\theta)}=\M_{G(1-\theta)}(n+1),
\end{equation*}
where we recall that $G(1-\theta)$ is a Gamma random variable with parameter $1-\theta$ whose law is denoted by $m$. Since both $I_{\phi}$ and $G(1-\theta)$ are moment determinate and so is $V_{\psi}$, see Theorem~\ref{thm:gL_main}\eqref{it:frakm}, we have
\begin{equation} \label{eq:decompos_gamma}
G(1-\theta)\eqindist \Vp \times I_{\phi},
\end{equation}
where $\eqindist$ stands for the identity in distribution and $\times$ represents the product of independent variables. Therefore, for any $f\in \Lg$, by H\"older's inequality and the factorization identity \eqref{eq:decompos_gamma}, we have
\begin{eqnarray} \label{eq:Ipbdd}
\|\Ip f\|_{\frakm}^2 &\leq &\int_0^{\infty}\Ip f^2(x)\frakm(x)dx = \int_0^{\infty}\int_0^{\infty}\iota(y)f^2(xy)dy\frakm(x)dx \\
&=& \int_0^{\infty}f^2(z)\int_0^{\infty}\frac1x \iota\left(\frac{z}{x}\right)\frakm(x)dxdz = \int_0^{\infty}f^2(z)m(z)dz=\|f\|^2_{m},
\end{eqnarray}
where the second last equality comes from the factorization \eqref{eq:decompos_gamma}. Therefore, we see that $\Ip \in\B(\Lg,\Lnu)$ with $|||\Ip|||\leq 1$. Next, for an arbitrary polynomial of order $n\in \mathbb{N}$, denoted by $\mathtt{p}_n(x)=\sum_{i=0}^n a_i x^i, a_i\in\R$, we write $g_n(x)=\sum_{i=0}^n \frac{a_i}{\M_{\Vp}(i+1)}x^i$. It is easy to observe that $g_n\in \Lg$ and $\Ip g_n(x)=f_n(x)$. Therefore, $\mathtt{p}_n\in Ran(\Ip)\subseteq \Lnu$. Using the fact that $V_{\psi}$ is moment determinate, we deduce that the set of polynomials are dense in $\Lnu$, see \cite[Corollary 2.3.3]{Akhiwzer_moment}, hence $\Ip$ has dense range in $\Lnu$. Moreover, as $\Ip$ is a Markov multiplier, i.e.~$\Ip \mathbf{1}(x)=\int_0^{\infty}\iota(y)dy=1$ where here $\mathbf{1}=\mathbf{1}_{\R_+}$. Furthermore, observe that
\[\Ip \mathbf{1}_{\{0\}}(x)=\int_0^{\infty}\iota(y)\mathbf{1}_{\{0\}}(xy)dy = \left\{\begin{array}{rcl} \int_0^{\infty} \iota(y)dy=1 & \mbox{if} & x=0, \\ 0 & \mbox{if} & x\neq 0, \end{array}\right.  \]
and hence $\Ip\mathbf{1}_{\{0\}}\equiv \mathbf{1}_{\{0\}}$. Moreover, for any $f\in \Lg$, $\Ip f(0)=\int_0^{\infty} f(0)\iota(y)dy = f(0)$. To prove similar results for $\hatLambda_{\phi}$, let us first observe that for any $f\in\Lg, g\in\Lnu, f,g\geq 0$,
\begin{eqnarray*}
\left\langle f, \hatLambda_{\phi} g\right\rangle_{m}&=&\left\langle \Ip f,g\right\rangle_{\frakm}= \int_0^{\infty}f(xy)\iota(y)dyg(x)\frakm(x)dx\\
&= &\int_0^{\infty}f(r)m^{-1}(r)\int_0^{\infty}\iota(r/x)g(x)\frakm(x)/xm(r)dr\\
&=& \int_0^{\infty}f(r)m^{-1}(r)\int_0^{\infty}g(rv)\frakm(rv)\iota(1/v)1/v dv m(r)dr.
\end{eqnarray*}
Moreover, for any $f\in\Lg, g\in\Lnu$, $|f|\in\Lg, |g|\in\Lnu$, hence we get that for any $g\in\Lnu$,
\begin{equation}\label{eq:hatLambda}
\hatLambda_{\phi} g(x)\eqae \frac{1}{m(x)}\int_0^{\infty}g(xy)\frakm(xy)\iota\left(\frac1y\right)\frac1y dy.
\end{equation}
Therefore, for any $x\geq 0$, $\hatLambda_{\phi}\mathbf{1}(x)= \frac{1}{m(x)}\int_0^{\infty}\frakm(xy)\iota\left(\frac1y\right)\frac1y dy=1$ by the factorization \eqref{eq:decompos_gamma}. Furthermore, both properties $\hatLambda_{\phi}\mathbf{1}_{\{0\}}=\mathbf{1}_{\{0\}}$ and $\hatLambda_{\phi}f(0)=f(0)$ can be proved using the same method as before. Next, we prove \eqref{eq:intertwin_refl_ss} in two steps. The first step is to establish \eqref{eq:intertwin_refl_ss} in $\Co$. Note that by identities \eqref{eq:PKrelation} and \eqref{eq:QRrelation}, in order to prove $P_t \Ip = \Ip Q_t$ on $\Co$, it suffices to show only that $\Pbar_t \Ip = \Ip \Qbar_t$ on $\Co$, for which we use the criteria stated in \cite[Proposition 3.2]{CPY_exp_stoch}.  On the one hand, by \eqref{eq:decompos_gamma}, we have
\begin{equation} \label{eq:factor_Gamma_z}
\M_{G(1-\theta)}(z)=\M_{\Vp}(z)\M_{I_{\phi}}(z)
\end{equation}
for all $z\in 1+i\R$. since $\M_{G(1-\theta)}(z)\neq 0$ on $z\in 1+i\R$ and $\M_{I_{\phi}}(z)<\infty $ on $z\in 1+i\R$, see \cite[Proposition 6.7]{PS_spectral}, we see from \eqref{eq:factor_Gamma_z} that $\M_{\Vp}(z)\neq 0$ on $z=1+i\R$. Hence by an application of the Wiener's Theorem, see e.g.~\cite[Lemma 7.9]{PS_spectral}, one concludes that the multiplicative kernel $\mathcal{V}_{\psi}$ associated to $\Vp$, i.e.~$\mathcal{V}_{\psi}f(x)=\E[f(x\Vp)]$, is injective on $\Co$. This combined with \eqref{eq:decompos_gamma} provides all conditions for the application of \cite[Proposition 3.2]{CPY_exp_stoch}, which gives that \eqref{eq:intertwin_refl_ss} holds for all $t\geq 0$ and $f\in\Co$. Next, recalling that $\Co\cap \Lm$ is dense in $\Lm$ (resp.~$\Co\cap \Lg$ is dense in $\Lg$), and since $\Ip \in \B(\Lg,\Lnu)$ and, for all $t\geq 0$, $\Pbar_t \in \B(\Lm), \Qbar_t \in \B(\Lm)$ (resp.~$P_t\in\Lnu,Q_t\in \Lg$), we conclude the extension of the intertwining relation between $\Pbar$ and $\Qbar$ from $\Co$ to $\Lm$ (resp.~between $P$ and $Q$ from $\Co$ to $\Lg$) by a density argument. Finally, using the properties of $\Ip$ proved in the first statement, we can directly apply Theorem~\ref{thm} to  deduce \eqref{eq:intertwin_killed_ss} from \eqref{eq:intertwin_refl_ss}. This concludes the proof of Theorem~\ref{thm:gL_main}\eqref{it:intertwin_ss}.

\subsection{Proof of Theorem~\ref{thm:gL_main}\eqref{it:Phi}} \label{sec:proof_phi}
In order to compute $\Phi_{\Ybar}$, we first note that \cite{Pitman_Yor_1996} has considered the normalization $\E_x[\tilde{\frakl}^R_t]=\int_0^t \overline{q}_s(x,0)ds$, where $\overline{q}_s(x,y)$ is the transition density of $\Qbar$ with respect to the speed measure $\mbar$. Under this normalization, we have
\[c(\mbar)=\lim_{t\rightarrow 0}\frac1t \int_0^t \int_0^{\infty}\mbar(x)\overline{q}_s(x,0)dxds = 1\]
where we used the property that the integration of $\overline{q}_s(x,0)$ with respect to the speed measure is 1. Hence by \cite[Section 5]{DM_ConstantsLocalTime}, we have, for $q>0$,
\[\Phi_{\Ybar}(q)=2\theta \tilde{\Phi}_R(q)=\frac{\Gamma(1-\theta)}{\Gamma(\theta)}2^{1-\theta}q^{\theta}.\]
Combining this formula with the intertwining relation $\Pbar_t \Lambda = \Lambda \Qbar_t$ and Theorem~\ref{thm}, we easily deduce that $\Phi_{\Xbar}=\Phi_{\Ybar}$ and this completes proof of the first half of Theorem~\ref{thm:gL_main}\eqref{it:Phi}. Now let us focus on computing $\Phi_X$ and $\Phi_Y$. As previously mentioned in Remark~\ref{remark_gL}(ii), $\tilde{\frakl}^Y$ is defined in \cite{HY_Inv_Local_time_OU} as the unique continuous increasing process such that
\begin{equation} \label{defn:LY_semimart}
\mathbf{N}_t=Y_t^{\theta}-\tilde{\frakl}^Y_t \quad \textnormal{is a martingale},
\end{equation}
which uses the Doob-Meyer decomposition of the semi-martingale $Y^{\theta}$, where we recall that $Y$ is the squared radial Ornstein-Uhlenbeck process of order $-\theta$. The expression of $\tilde{\Phi}_Y$, the Laplace exponent of the inverse of $\tilde{l}^Y$, is given in \eqref{eq:Phi_Y_HY}.
% and we see that $\delta_Y=\lim_{q\rightarrow 0}\Phi_Y(q)=0$, $\gamma_Y=\lim_{q\rightarrow 0}\frac{\Phi_Y(q)}{q}=0$, and its L\'evy measure is given by
%\begin{equation}
%\mu_Y(y)= \frac{1}{\Gamma(\theta)}\frac{e^{-\theta y}}{(1-e^{-y})^{1+\theta}}
%=\frac{1}{\Gamma(\theta)}\sum_{n=0}^{\infty}\frac{\Gamma(\theta+n+1)}{\Gamma(1+\theta)\Gamma(n+1)}e^{-(n+\theta)y}
%=\int_{0}^{\infty}e^{-sy}\sigma^Y(ds)
%\end{equation}
%where
%\begin{equation} \label{eq:sigma_Y}
%\sigma^Y(ds)=\frac{1}{\Gamma(\theta)}\sum_{n=0}^{\infty}\frac{\Gamma(\theta+n+1)}{\Gamma(1+\theta)\Gamma(n+1)}\delta_{(n+\theta)}(s)ds.
%\end{equation}
Therefore, our goal is to compute the constants $\tilde{c}(\frakm)$ and  $\tilde{c}(m)$ and we simply have,
\[ \Phi_X(q)=\frac{\tilde{\Phi}_X(q)}{\tilde{c}(\frakm)},\quad  \Phi_Y(q)=\frac{\tilde{\Phi}_Y(q)}{\tilde{c}(m)}.\]
In this direction, we will need the following Lemma, which is a generalization of \cite[Proposition 2.1]{HY_Inv_Local_time_OU} from continuous semi-martingales to c\`adl\`ag semi-martingales, and serves as a stepping stone for computing $\tilde{c}(\frakm)$.
\begin{lemma}\label{lemma:Local_time_transform}
Let $(M_t)_{t\geq 0}$ be a c\`adl\`ag semi-martingale with $M_0=0$. Let $g:\R_{+}\rightarrow\R_{+}$ be an increasing continuous function with $g(0)=0$, and let $h:\R_{+}\rightarrow\R_{+}$ be a strictly positive, continuous function, locally with bounded variation. We set
\begin{equation*}
N_t=h(t)M_{g(t)},\quad t\geq 0,
\end{equation*}
and we denote by $\tilde{\frakl}^M$ (resp.~$\tilde{\frakl}^N$) the local time at 0 of the c\`adl\`ag semi-martingale $M$ (resp.~$N$). Then $\tilde{\frakl}^N$ can be obtained from a simple transform of $\tilde{\frakl}^M$ by
\begin{equation}
\tilde{\frakl}^N_t=\int_0^t h(s)d\tilde{\frakl}^M_{g(s)}.
\end{equation}
\end{lemma}
\begin{proof}
By definition of the local time via the Meyer-Tanaka formulae, see \cite[Chapter IV]{Protter_Stoc_Integ}, one has
\begin{eqnarray}
|M_t|&=&\int_0^{t}sgn(M_s)dM_s+\tilde{\frakl}^M_t+\sum_{0<s\leq t}(|M_s|-|M_{s-}|-sgn(M_{s-})\Delta M_s),\\
|N_t|&=&\int_0^{t}sgn(N_s)dN_s+\tilde{\frakl}^N_t+\sum_{0<s\leq t}(|N_s|-|N_{s-}|-sgn(N_{s-})\Delta N_s), \label{eq:Localtime_N}
\end{eqnarray}
where the function $sgn$ is the sign function defined by $sgn(x)=\mathbf{1}_{\{x>0\}} -\mathbf{1}_{\{x<0\}}$. Consequently,
\begin{align*}
|M_{g(t)}|=&\int_0^{g(t)}sgn(M_s)dM_s+\tilde{\frakl}^M_{g(t)}+\sum_{0<s\leq g(t)}(|M_s|-|M_{s-}|-sgn(M_{s-})\Delta M_s)\\
=&\int_0^{t}sgn(N_s)d((h(s))^{-1}N_s)+\tilde{\frakl}^M_{g(t)}+\sum_{0<s\leq t}(h(s))^{-1}(|N_s|-|N_{s-}|-sgn(N_{s-})\Delta N_s)\\
=&\int_0^{t}sgn(N_s)(h(s))^{-1}dN_s-\int_0^{t}(h(s))^{-2}|N_s|dh(s)+\tilde{\frakl}^M_{g(t)}\\
&+\sum_{0<s\leq t}(h(s))^{-1}(|N_s|-|N_{s-}|-sgn(N_{s-})\Delta N_s).
\end{align*}
Therefore using integration by parts, we have
\begin{eqnarray}
|N_t|&=& h(t)|M_{g(t)}| = \int_0^t h(s)dM_{g(s)}+\int_0^t M_{g(s)}dh(s) \nonumber\\
&=& \int_0^{t}sgn(N_s)d(N_s)-\int_0^{t}(h(s))^{-1}|N_s|dh(s)+\int_0^t h(s)d\tilde{\frakl}^M_{g(s)}+\int_0^{t}(h(s))^{-1}|N_s|dh(s)\nonumber\\
& &+\sum_{0<s\leq t}(|N_s|-|N_{s-}|-sgn(N_{s-})\Delta N_s)\nonumber\\
&=& \int_0^{t}sgn(N_s)dN_s+\int_0^t h(s)d\tilde{\frakl}^M_{g(s)}+\sum_{0<s\leq t}(|N_s|-|N_{s-}|-sgn(N_{s-})\Delta N_s),\label{eq:Localtime_N2}
\end{eqnarray}
which, by identification between \eqref{eq:Localtime_N} and \eqref{eq:Localtime_N2}, yields that $\tilde{\frakl}^N_t=\int_0^th(s)d\tilde{\frakl}^M_{g(s)}$.
\end{proof}
Now let us compute the constants $\tilde{c}(\frakm)$ and $\tilde{c}(m)$. To this end, we first recall from \cite{rivero} that $p_{\theta}(x)=x^{\theta},x>0,$ is an invariant function for the semigroup $\Pbar^{\dag}$, therefore $\Pbar_tp_{\theta}(x)\geq \Pbar^{\dag}_tp_{\theta}(x)=p_{\theta}(x)$, from which we deduce that the process $(\Xbar^{\theta})=(\Xbar^{\theta}_t)_{t\geq 0}$ is a submartingale. Hence using a similar definition as \eqref{defn:LY_semimart}, we define $\tilde{\frakl}^{\Xbar}$ as the unique increasing process such that
\begin{equation}
\mathbf{M}_t=\Xbar_t^{\theta}-\tilde{\frakl}^{\Xbar}_t \quad \textnormal{is a martingale}.
\end{equation}
Using the deterministic time change \eqref{eq:YZ_relation} between $X$ and $\Xbar$, we get $X_t^{\theta}=e^{-\theta t}\Xbar_{e^t-1}^{\theta}$, hence Lemma~\ref{lemma:Local_time_transform} yields that
\begin{eqnarray*}
\tilde{\frakl}^X_t &=& \int_0^t e^{-\theta s}d\tilde{\frakl}^{\Xbar}_{e^s-1}=\int_0^te^{-\theta s}\left(d\Xbar^{\theta}_{e^s-1}+d\mathbf{M}_{e^s-1}\right)=\int_0^te^{-\theta s}d(e^{\theta s}X_s^{\theta})+\int_0^t e^{-\theta s}d\mathbf{M}_{e^s-1}\\
&=& \theta\int_0^t X_s^{\theta}ds+X_t^{\theta}-X_0^{\theta}+\int_0^t e^{-\theta s}d\mathbf{M}_{e^s-1}.
\end{eqnarray*}
Now we observe that, on the one hand,
\begin{align*}
\int_0^{\infty}\E_x\left[\int_0^t X_s^{\theta}ds\right]\frakm(x)dx&=\int_0^t\int_0^{\infty}\E_x\left[ X_s^{\theta}\right]\frakm(x)dxds=\int_0^t\frakm P_s p_{\theta}ds\\
&= \int_0^t\frakm p_{\theta}ds=\frac{W_{\phi}(1+\theta)}{\Gamma(1-\theta)\Gamma(1+\theta)},
\end{align*}
where we used the fact that $\frakm(x)dx$ is an invariant measure for the semigroup $P$. On the other hand, by the martingale property of $(\mathbf{M}_t)_{t\geq 0}$, we have $\E_x[\int_0^t e^{-\theta s}d\mathbf{M}_{e^s-1}]=0$ for all $x\geq 0$. Hence, by the definition of $\tilde{c}(\frakm)$, see \eqref{eq:c_nu}, and the definition of semigroup $P$, we get
\begin{align*}
\tilde{c}(\frakm)&=\int_0^{\infty}\E_x[\tilde{\frakl}^X_1]\frakm(x)dx=\int_0^{\infty}\E_x\left[\left(\theta\int_0^1 X_s^{\theta}ds+X_1^{\theta}-X_0^{\theta}\right)\right]\frakm(x)dx\\
&=\frac{\theta W_{\phi}(1+\theta)}{\Gamma(1-\theta)\Gamma(1+\theta)}+\frakm P_1p_{\theta}-\frakm p_{\theta}=\frac{\theta W_{\phi}(1+\theta)}{\Gamma(1-\theta)\Gamma(1+\theta)}.
\end{align*}
In particular, since $\phi^Y(u)=u$, we have $\tilde{c}(m)=\frac{\theta}{\Gamma(1-\theta)}$, and  Theorem~\ref{thm:gL_main}\eqref{it:Phi} follows from dividing \eqref{eq:Phi_Y_HY} by $\tilde{c}(m)$.
\subsection{Proof of Theorem~\ref{thm:gL_main}\eqref{it:Krein} and spectral expansions} \label{sec:spectral}
In the section, we will prove Theorem~\ref{thm:gL_main}\eqref{it:Krein} by providing the spectral expansion of $P_tf$ and $P_t^{\dag}f$. In fact, we will find conditions on $\psi$, $f$ and $t$ such that these expansions hold. Note that the expansions for $\Pbar$ and $\Pbar^{\dag}$ require additional analysis that will be detailed in a forthcoming paper, see already the paper by Patie and Zhao \cite{Patie_Zhao_17}, which provides the spectral expansions for reflected stable processes. Let us start by recalling some well-known results for the self-adjoint semigroups $Q$ and $Q^{\dag}$. For $n\geq 0$, let $\Lag_n$ and $\Lag^{\dag}_n$ be the Laguerre polynomials  (of different orders) defined by
\begin{eqnarray}
\Lag_n(x) &= & \frac{\mathcal{R}^{(n)}m(x)}{m(x)}=\sum_{k=0}^{n}(-1)^k \frac{\Gamma(n+1-\theta)}{\Gamma(k+1-\theta)\Gamma(n-k+1)} \frac{x^k}{k!},\label{eq:Laguerre}\\
\Lag^{\dag}_n(x) &=& \sum_{k=0}^{n}(-1)^k \frac{\Gamma(n+1+\theta)}{\Gamma(k+1+\theta)\Gamma(n-k+1)} \frac{x^{k+\theta}}{k!},\label{eq:Laguerre_dag}
\end{eqnarray}
where $\mathcal{R}^{(n)}f(x)=\frac{(x^nf(x)) ^{(n)}}{n!}$ is the Rodrigues operator. Then $\Lag_n\in \Lg$ (resp.~$\Lag^{\dag}_n\in \Lg$) is an eigenfunction of $Q_t$ (resp.~$Q^{\dag}_t$) associated with eigenvalue $e^{-nt}$ (resp.~$e^{-(n+\theta)t}$), i.e.~$Q_t \Lag_n(x)=e^{-nt}\Lag_n(x)$ (resp.~$Q^{\dag}_t \Lag^{\dag}_n(x)=e^{-(n+\theta)t}\Lag^{\dag}_n(x)$) for all $n\geq 0$. Moreover, for any $t>0, f\in\Lg$, $Q_t$  and $Q_t^{\dag}$ admit the following spectral expansions in $\Lg$
\begin{eqnarray}
Q_tf &=& \sum_{n=0}^{\infty} e^{-nt} \frakc_{n}(-\theta)\left\langle f,\Lag_n \right\rangle_{m} \Lag_n,\label{eq:spectral_Q}\\
Q^{\dag}_tf &= & \frac{\Gamma(1-\theta)}{\Gamma(1+\theta)}\sum_{n=0}^{\infty} e^{-(n+\theta)t} \frakc_{n}(\theta)\left\langle f,\Lag^{\dag}_n \right\rangle_{m} \Lag^{\dag}_n,\label{eq:spectral_Q_dag}
\end{eqnarray}
where for any $n\geq 0, u>-1$, we set
\begin{equation} \label{eq:frakc_nm}
\frakc_{n}(u) = \frac{\Gamma(1+u)\Gamma(n+1)}{\Gamma(n+1+u)}.
\end{equation}
In order to study the spectral expansions of $P$ and $P^{\dag}$, we again recall from \cite{rivero} that the function $p_{\theta}(x)=x^{\theta}$ is an invariant function for semigroup $\Pbar^{\dag}$. Hence we have
\begin{equation*}
P^{\dag}_tp_{\theta}(x)=\Pbar^{\dag}_{e^t-1}\rmd_{e^{-t}}p_{\theta}(x)=\Pbar^{\dag}_{1-e^{-t}}p_{\theta}(xe^{-t})=p_{\theta}(xe^{-t})=e^{-\theta t}p_{\theta}(x),
\end{equation*}
i.e.~$p_{\theta}$ is a $\theta$-invariant function for semigroup $P^{\dag}$. Therefore, by Doob's $h$-transform, we can define a semigroup $P^{\uparrow}=(P_t^{\uparrow})_{t\geq 0}$, for $t\geq 0$ and $x>0$, by
\begin{equation}\label{eq:h_transform_theta_dag}
P^{\uparrow}_tf(x)=e^{\theta t}\frac{P^{\dag}_t p_{\theta}f(x)}{p_{\theta}(x)}.
\end{equation}
Note that $P^{\uparrow}$ is a generalized Laguerre semigroup associated to $\psi_{\uparrow} \in\N_{\uparrow}$, which we recall is defined as $\psi_{\uparrow}(u)=\psi(u+\theta)$ for all $u\geq 0$. Therefore, as shown in \cite{PS_spectral}, the semigroup $P^{\uparrow}$ has an invariant measure $m^{\uparrow}$, whose law is absolutely continuous and determined by its entire moments
\begin{equation}
\M_{m^{\uparrow}}(n+1)=\frac{\prod_{k=1}^n \psi_{\uparrow}(k)}{n!}, n\in \mathbb N.
\end{equation}
Next,  we say that a  sequence $(P_n)_{n\geq 0}$ in the Hilbert space $\Lnu$ is a Bessel sequence  if there exists   $A>0$ such that
\begin{equation} \label{eq:frame1}
 \sum_{n=0}^{\infty}  |\langle f, P_n\rangle_\nu |^2 \leq A   ||f||^2_\nu
\end{equation}
hold, for all $f \in \Lnu$, see e.g.~the monograph \cite{Christensen-03}. The constant $A$ is called a Bessel bound.
Recalling that the class $\N$ is defined as the collection of $\psi$ in the form \eqref{eq:psi}, we further define the following subclasses of $\N$. Denoting $\overline{\overline{\Pi}}(y)=\int_{y}^{\infty}\int_{r}^{\infty}\Pi(dx)dr$  the double tail of $\Pi$, we set
\begin{align}
\N_P &= \{\psi\in \N; \sigma^2>0 \},\\
\overline{\N}_{\infty} &= \N_P \cup \{\psi\in \N; \sigma^2=0, \overline{\overline{\Pi}}(0+) = \infty\}.
\end{align}
Note that when $\psi \in \overline{\N}_{\infty}$ then $\lim_{u\to \infty }\frac{\psi(u)}{u}=\infty$.
Moreover, define the following sets of $(\psi,f)$,
\begin{eqnarray}
\D^{\checkmark}(\Ip)&=& \{(\psi,f);\: \psi\in \N_{\checkmark},\: f\in Ran(\Ip)\},\\
\D^{\N_P}(\frakm)&=& \{(\psi,f);\: \psi\in \N_P\cap \N_{\checkmark}, \: f\in \Lnu\}.
\end{eqnarray}
Finally, for any $\psi \in \N$, we let
\begin{equation}\label{eq:Poly_n}
\Poly_n^{\psi}(x) = \sum_{k=0}^{n}(-1)^k {n\choose k} \frac{k!}{\prod_{i=1}^k\psi(i)}x^k.
\end{equation}
We are now ready to state the following theorem, which provides spectral properties of the non-self-adjoint semigroups $P_tf$ and $P^{\dag}_tf$.
\begin{theorem} \label{thm:spectral}
For any $\psi\in \N_{\checkmark}$, we have the following.
\begin{enumerate}
\item \label{it:eigen} Let us write, for any $n\in \mathbb{N}$,
\begin{equation}
\Poly_n(x) = \Poly_n^{\psi}(x), \quad \Poly^{\dag}_n(x) =x^{\theta} \Poly_n^{\psi_{\uparrow}}(x).
\end{equation}
Then $\Poly_n\in \Lnu$ (resp.~$\Poly^{\dag}_n\in \Lnu$) is an eigenfunction of $P_t$ (resp.~$P^{\dag}_t$) associated to the eigenvalue $e^{-nt}$ (resp.~$e^{-(n+\theta)t}$). Moreover,  the sequence $\left(\frakc^{-\frac12}_n(-\theta)\Poly_n\right)_{n\geq 0}$ is a dense Bessel sequence in $\Lnu$ with upper bound 1, where we recall that $\frakc_n(u)$ is defined in \eqref{eq:frakc_nm}. Finally, we have $(e^{-nt})_{n\geq 0}= S(Q_t)\subseteq S(P_t)$, and $(e^{-(n+\theta)t})_{n\geq 0}= S(Q^{\dag}_t)\subseteq S(P^{\dag}_t)$.
\item \label{it:coeigen} For any $\psi\in \N_{\checkmark}\cap \overline{\N}_{\infty}$ and $n\geq 0$, let
\begin{equation} \label{eq:nu_n}
\frakm_n(x)=\frac{\Rcal^{(n)}\frakm(x)}{\frakm(x)}, \quad \frakm^{\dag}_n(x) = \frac{\Rcal^{(n)}m^{\uparrow}(x)}{x^{\theta}\frakm(x)}.
\end{equation}
Then $\frakm_n$ (resp.~$\frakm^{\dag}_n$)
% is the unique solution to the equation
%\begin{eqnarray}
%\hatLambda_{\phi}f &=& \Lag_n, \label{eq:hatIp_Lag}\\
%(resp.~\hatLambda_{\phi}f &=&\frac{\Gamma(1-\theta)}{W_{\phi}(1+\theta)}\Lag^{\dag}_n,) \label{eq:hatIp_Lag_dag}
%\end{eqnarray}
%and
is an eigenfunction of $\hatP_t$ (resp.~$\hatP^{\dag}_t$) associated to the eigenvalue $e^{-nt}$ (resp.~$e^{-(n+\theta)t}$). Moreover, the sequences $(\Poly_n)_{n\geq 0}$ and $(\frakm_n)_{n\geq 0}$ (resp.~$(\Poly^{\dag}_n)_{n\geq 0}$ and $(\frakm^{\dag}_n)_{n\geq 0}$) are biorthogonal sequences in $\Lnu$. Furthermore, if $\psi\in \N_P\cap \N_{\checkmark}$, then for any $\epsilon>0$ and large $n$,
\begin{equation} \label{eq:nu_n_estimate_NP}
\|\frakm_n\|_{\frakm}=O(e^{\epsilon n}).
\end{equation}
If in addition $\overline{\overline{\Pi}}(0+)<\infty$, then with $\frakb=\frac{\beta+\overline{\overline{\Pi}}(0+)}{\sigma^2}$, we have for large $n$,
\begin{equation} \label{eq:nu_n_estimate_frakm}
\|\frakm_n\|_{\frakm}=O(n^{\frakb}),
\end{equation}
and the sequence $(\sqrt{\frakc_n(\frakb)}\frakm_n)_{n\geq 0}$ is a Bessel sequence in $\Lnu$ with bound 1.
\item \label{it:spec} For any $t>0$ and $(\psi,f)\in \D^{\checkmark}(\Ip) \cup \D^{\N_P}(\frakm)$, we have in $\Lnu$ the following spectral expansions
\begin{eqnarray}
P_t f(x) &=& \sum_{n=0}^{\infty} e^{-nt}\left\langle  f,\frakm_n \right\rangle_{\frakm}\Poly_n(x),\label{eq:spectral_P}\\
P^{\dag}_t f(x) &=& \sum_{n=0}^{\infty} e^{-(n+\theta)t}\left\langle  f,\frakm^{\dag}_n \right\rangle_{\frakm}\Poly^{\dag}_n(x).\label{eq:spectral_P_dag}
\end{eqnarray}
\end{enumerate}
\end{theorem}
Before proving the previous Theorem,  we state the following corollary which  gives the speed of convergence to equilibrium in the Hilbert space topology $\Lnu$.
\begin{corollary}\label{thm:convergence}
Let $\psi\in\N_P \cap \N_{\checkmark}$ with $\overline{\overline{\Pi}}(0+)<\infty $, then recalling that $\frakb=\frac{\beta+\overline{\overline{\Pi}}(0+)}{\sigma^2}$, we have, for any $f\in \Lnu$ and $t>0$,
\begin{equation}
\|P_tf-\frakm f\|_{\frakm}\leq \sqrt{\frac{\frakb+1}{1-\theta}}e^{-t}\|f-\frakm f\|_{\frakm}.
\end{equation}
\end{corollary}
The rest of this section is devoted to the proof of these results.
\subsubsection{Proof of Theorem~\ref{thm:spectral}\eqref{it:eigen}}
Let $\psi\in\N_{\checkmark}$ and recall that $\Ip p_k(x) = \E[x^k I_{\phi}^k] = \frac{k!}{a_k(\phi)} p_k(x)$. Use the linearity of $\Ip$ and note that for any $n\geq 0$,
\begin{align*}
\Ip \Lag_n(x) &= \sum_{k=0}^{n} \frac{(-1)^k\Gamma(n+1-\theta)}{\Gamma(k+1-\theta)\Gamma(n-k+1)} \frac{1}{k!} \Ip p_k(x)\\ & =\sum_{k=0}^{n}\frac{(-1)^k (n-\theta)\dots(k+1-\theta)}{(n-k)!} \frac{1}{\prod_{i=1}^{k} \phi(i)} p_k(x)\\
&=\sum_{k=0}^{n}(-1)^k \frac{(n-\theta)\dots(k+1-\theta)}{(n-k)!} \prod_{i=1}^{k} \frac{i-\theta}{\psi(i)} p_k(x)={n-\theta \choose n} \sum_{k=0}^{n} {n\choose k} \frac{(-1)^kk!}{\prod_{i=1}^k \psi(i)}x^k\\ &=\frac{\Poly_n(x)}{\frakc_n(-\theta)}.
\end{align*}
Since $\Lag_n\in \Lg$, and $\Ip\in \B(\Lg,\Lnu)$, we get that $\Poly_n\in \Lnu$. Apply the intertwining relation \eqref{eq:intertwin_refl_ss}, together with $Q_t\Lag_n(x)=e^{-nt}\Lag_n(x)$, we get, for each $n\in \mathbb{N}$,
\begin{equation*}
P_t \Poly_n(x) = \frakc_n(-\theta)P_t\Ip\Lag_n(x) =\frakc_n(-\theta)\Ip Q_t\Lag_n(x)=\frakc_n(-\theta) e^{-nt}\Ip \Lag_n(x)=e^{-nt}\Poly_n(x).
\end{equation*}
This proves the eigenfunction property of $\Poly_n$. Next, using the fact that $V_{\psi}$ is moment determinate, we see that the set of polynomials are dense in $\Lnu$, see \cite[Corollary 2.3.3]{Akhiwzer_moment}, which proves the completeness of $(\Poly_n)_{n\geq 0}$. Next, to get the Bessel property of $\left(\frakc^{-\frac12}_n(-\theta)\Poly_n\right)_{n\geq 0}$, we observe that, for any $f\in \Lnu$,
\begin{align*}
\sum_{n=0}^{\infty}\left|\left\langle f,\frakc^{-\frac12}_n(-\theta)\Poly_n\right\rangle_{\frakm}\right|^2&=\sum_{n=0}^{\infty}\left|\left\langle f,\sqrt{\frakc_n(-\theta)}\Ip \Lag_n\right\rangle_{\frakm}\right|^2=\sum_{n=0}^{\infty}\left|\left\langle \hatIp f,\sqrt{\frakc_n(-\theta)}\Lag_n\right\rangle_{\frakm}\right|^2 \\
&=\| \hatIp f \|^2_{m}\leq \|f\|^2_{\frakm},
\end{align*}
where we used the Parseval identity for the (normalized) Laguerre polynomials in $\Lg$, see e.g.~\cite[Section 2.7]{Bakry_2014}, and the fact that $\hatIp\in \B(\Lnu,\Lg)$ as the adjoint of $\Ip\in \B(\Lg,\Lnu)$ with $|||\hatIp|||=|||\Ip|||\leq 1$. Finally, using similar computations than above, we observe that $\Poly^{\dag}_n = \frac{W_{\phi}(1+\theta)}{\Gamma(1+\theta)}\frakc_n(\theta)\Ip \Lag^{\dag}_n$, and the proof for $\Poly^{\dag}_n$ being an eigenfunction for $P^{\dag}_t$ with eigenvalue $e^{-(n+\theta)t}$ follows through a similar line of reasoning using the intertwining relation with $Q^{\dag}_t$. This concludes the proof.

\subsubsection{Proof of Theorem~\ref{thm:spectral} \eqref{it:coeigen}}
%\item If $\psi\in \overline{\N}_{\infty}^c \cap \N_{\checkmark}$, recalling that $N_{\frakr}=\lceil \frac{\overline{\Pi}(0+)}{\frakr}\rceil -1 \in [0,\infty]$, we define, for any $n<\max\left(\lceil \frac{N_{\frakr}}{2}\rceil-1,0\right)$, $\frakm_n$  as in \eqref{eq:defn_nun}, then we have
%\[\frakm_n\in \Lnu \cap \mathtt{C}^{N_{\frakr}-n}((0,\frakr))\]
%and $\frakm_n$ is a co-eigenfunction associated to the eigenvalue $e^{-nt}$.
%In order to prove Theorem~\ref{thm:spectral}\eqref{it:coeigen}, we start with the following setup.
Let us write $\T_1\psi(u)=\frac{u}{u+1}\psi(u+1)$ for $u>0$, then by \cite[Lemma 2.1]{Patie_Kyprianou_2011}, $\T_1\psi$ is the Laplace exponent of a spectrally negative L\'evy process, which satisfies $\T_1\psi(0)=0$ and $(\T_1\psi)^{'}(0)=\psi(1)>0$. Hence $\T_1\psi \in \N_{\uparrow}$ and therefore by \cite[Theorem 1.5]{PS_spectral}, $\T_1\psi$ characterizes a generalized Laguerre semigroup, denoted by $\ttP=(\ttP_t)_{t\geq 0}$, with an invariant measure denoted by $\breve{\frakm}$, and the spectral properties of $\ttP$ have been studied in \cite{PS_spectral}. In the rest of the paper, this semigroup $\ttP$ will serve as a reference semigroup in order for us to develop further spectral results for $P$.
Our first aim is to establish an intertwining relation between the semigroups $P$ and $\ttP$. To this end, we need introduce a few objects and notation. Let $\Z$ be a random variable whose law is given by
\begin{equation}
\Prob(\Z\in dx) = \psi(1)W_{+}'(-\ln x)dx+W(0)\delta_1(x), \quad x\in[0,1],
\end{equation}
with $\delta_1$ denoting the Dirac mass at 1, and $W_{+}^{'}$ being the right-derivative of the so-called scale function of the L\'evy process $\xi$, see e.g.~\cite[Section 8.2]{Kyprianou}, which is an increasing function $W:[0,\infty) \longrightarrow [0,\infty)$ characterized by its Laplace transform
\begin{equation} \label{eq:defn_W}
\int_{0}^{\infty}e^{-\lambda x}W(x)dx = \frac{1}{\psi(\lambda)}, \quad \lambda >0.
\end{equation}
We also recall that $W(0)=0$ whenever $\psi \in \overline{\N}_{\infty}$ and thus in such case the law of $\Z$ is absolutely continuous with a density denoted by $\mathfrak{z}$.
We are now ready to state and prove the following lemma.
\begin{lemma}
Define the multiplicative kernel $\Lambda_{\Z}$ as $\Lambda_{\Z} f(x)=\E[f(x\Z)]$, then $\Lambda_{\Z} \in \B(\Co)\cap\B(\Lnu,\rmL^2(\breve{\frakm}))$ with $|||\Lambda_{\Z}||| \leq 1$. Furthermore, for all $f\in\Lnu$, we have
\begin{equation} \label{eq:Intertwin_P_P1}
\Lambda_{\Z} P_tf = \ttP_t \Lambda_{\Z} f.
\end{equation}
\end{lemma}
\begin{proof}
First, we observe that, for all $n\in \mathbb{N}$,
\begin{equation}
\M_{\Vp}(n+1) = \frac{\prod_{k=1}^n \psi(k)}{n!}=\frac{\prod_{k=1}^{n}\frac{k}{k+1}\psi(k+1) }{n!}\frac{\psi(1)(n+1)}{\psi(n+1)}=\M_{V_{\T_1\psi}}(n+1) \frac{\psi(1)(n+1)}{\psi(n+1)},
\end{equation}
where, by \cite[Theorem 2.1]{PS_spectral}, $V_{\T_1\psi}$ is the random variable whose law is the stationary distribution of $\ttP$ and is determined by its entire moments $\M_{V_{\T_1\psi}}(n+1)=\frac{\prod_{k=1}^n \T_1\psi(k)}{n!}$. Now by \eqref{eq:defn_W}, we have, using an obvious change of variable and integration by parts, that for each $n\in \mathbb{N}$,
\begin{equation*}
\frac{1}{\psi(n+1)} = \int_0^{\infty}e^{-(n+1)x}W(x)dx=\int_0^1u^nW(-\ln u)du =\frac{1}{n+1}\left(W(0)+\int_0^1u^nW_{+}^{'}(-\ln u)du\right).
\end{equation*}
Therefore,
\begin{align*}
\M_{\Vp}(n+1) &=\M_{V_{\T_1\psi}}(n+1) \frac{\psi(1)(n+1)}{\psi(n+1)}=\M{V_{\T_1\psi}}(n+1)\psi(1)\int_0^1u^nW_{+}^{'}(-\ln u)+W(0)\delta_1(u)du \\
&= \M_{V_{\T_1\psi}}(n+1)\int_0^1u^n \zeta(u)du=\M_{V_{\T_1\psi}}(n+1)\M_{\Z}(n+1).
\end{align*}
Both variables $\Vp$ and $V_{\T_{1}\psi}$ are moment determinate by Theorem~\ref{thm:gL_main}\eqref{it:frakm} and \cite[Theorem 2.1]{PS_spectral}, and so does $\Z$ since it has compact support. Hence we conclude that
\begin{equation}\label{decomp_V1Z}
\Vp \eqindist V_{\T_{1}\psi} \times \Z.
\end{equation}
Therefore, the facts that $\Lambda_{\Z}\in \B(\Lnu,\rmL^2(\breve{\frakm}))$ and $|||\Lambda_{\Z}|||\leq 1$ follow from similar arguments as \eqref{eq:Ipbdd} and $\Lambda_{\Z}\in \B(\Co)$ follows easily from dominated convergence. Moreover, by \cite[Lemma 7.9]{PS_spectral}, the multiplicative kernel $\mathcal{V}_{\T_1\psi}$ defined by $\mathcal{V}_{\T_1\psi}f(x)=\E[f(x V_{\T_{1}\psi})]$ is one-to-one in $\Co$. Hence again using \cite[Proposition 3.2]{CPY_exp_stoch}, the intertwining relation \eqref{eq:Intertwin_P_P1} holds for all $f\in\Co$, and we can further extend this relation to $\Lnu$ using a density argument as $\Co \cap \Lnu$ is dense in $\Lnu$ and the fact that $P_t\in \Lnu$, $\ttP_t\in \rmL^2(\breve{\frakm})$. This completes the proof.
\end{proof}
\begin{corollary} \label{cor:supp_smoo_m}
For any $\psi\in \overline{\N}_{\infty}\cap \N_{\checkmark}$, we have $\frakm(x)>0$ for any $x>0$ and $\frakm\in \Ci_0(\R_+)$.
\end{corollary}
\begin{proof}
Let us write $\phi_1(u)=\frac{\T_1\psi(u)}{u},u\geq 0$, then since $\T_1\psi \in \N_{\uparrow}$, an application of the Wiener-Hopf factorization yields that $\phi_1$ is a Bernstein function, see \cite[(1.8)]{PS_spectral}. Moreover, by observing that $\phi_1(u)=\frac{u+1-\theta}{u+1}\phi(u+1)$, it is easy to see that $\lim_{u\to \infty}\phi_1(u)=\phi(u)=\infty$ as $\psi \in \overline{\N}_{\infty}$. Hence  by \cite[Theorem 1.6]{PS_spectral}, the density of $\breve{\frakm}$ is concentrated and positive on $(0,\infty)$. Now since, for all $n\in \mathbb{N}$
\begin{equation*}
\E[V_{\psi}^{n+1}]=\frac{\prod_{k=1}^{n+1}\psi(k)}{(n+1)!}=\psi(1)\frac{\prod_{k=1}^n\T_1\psi(k)}{n!} = \psi(1)\E[V_{\T_1\psi}^n],
\end{equation*}
we get by moment determinacy that
\begin{equation} \label{eq:relation_nu_nubar}
x\frakm(x)=\psi(1)\breve{\frakm}(x),\quad x>0.
\end{equation}
This implies that the density of $\frakm$ has the same support as $\breve{\frakm}$. Now let $\Pi_1$ denote the L\'evy measure of $\T_1\psi$, then by \cite[Theorem 2.2]{PP_refined_factorization},
\begin{equation} \label{eq:Pi_Pione}
\overline{\Pi}_1(y)=\int_y^{\infty}(e^{-r}\overline{\Pi}(r)dr+e^{-r}\Pi(dr))=e^{-y}\overline{\Pi}(y), \quad \overline{\Pi}_1(0+)=\overline{\Pi}(0+),
\end{equation}
therefore if $\psi\in\overline{\N}_{\infty}$, so does $\T_1\psi$ and therefore $\breve{\frakm} \in \mathtt{C}_0^{\infty}(\R_+)$ by \cite[Theorem 2.5]{PS_spectral}. Again using
\eqref{eq:relation_nu_nubar}, $\frakm$ and $\breve{\frakm}$ have the same smoothness properties, which shows that $\frakm\in \Ci_0(\R_+)$.
\end{proof}
We now have all the ingredients to prove Theorem~\ref{thm:spectral}\eqref{it:coeigen}. From \eqref{eq:Pi_Pione}, it is easy to see that if $\psi\in\overline{\N}_{\infty} \cap \N_{\checkmark}$, then $\T_1\psi \in \overline{\N}_{\infty}\cap \N_{\uparrow}$ and we see from \cite[Theorem 2.19]{PS_spectral} that $\ttP_t$ has co-eigenfunctions $\breve{\frakm}_n\in {\rm{L}}^2(\breve{\frakm})$, given by $\breve{\frakm}_n(x)=\frac{\mathcal{R}^{(n)}\breve{\frakm}(x)}{\breve{\frakm}(x)}$.
Now let us define, for any $n\in \mathbb{N}$,
\begin{equation} \label{coeigen_nu}
\frakm_n =\widehat{\Lambda}_{\Z}\breve{\frakm}_n,
\end{equation}
then $\frakm_n\in \Lnu$ since $\widehat{\Lambda}_{\Z} \in \B(\rmL^2(\breve{\frakm}),\Lnu)$.
%Hence $\frakm_n$ is a co-eigenfunction for $P_t$.
Moreover, similar to \eqref{eq:hatLambda}, we deduce that, for almost every (a.e.) $x>0$,
\begin{equation}
\frakm_n(x) =\widehat{\Lambda}_{\Z} \breve{\frakm}_n(x) = \frac{1}{\frakm(x)} \int_0^{\infty} y^{-1} \breve{\frakm}_n(xy)\breve{\frakm}(xy)\mathfrak{z}\left(\frac1y\right)dy= \frac{1}{\frakm(x)} \int_0^{\infty} y^{-1} \mathcal{R}^{(n)}\breve{\frakm}(xy)\mathfrak{z}\left(\frac1y\right)dy,
\end{equation}
where we recall that $\mathfrak{z}$ denotes the density of the random variable $\Z$ whose law  is absolutely continuous as $W(0)=0$ with $\psi \in \overline{\N}_{\infty}$. We write, for  any $n\in \mathbb{N}$,  $w_n(x)=\frakm_n(x)\frakm(x)$ and $\wbar_n(x)=\breve{\frakm}_n(x)\breve{\frakm}(x)=\mathcal{R}^{(n)}\breve{\frakm}(x), x>0$, then the above equation is equivalent to
\begin{equation}\label{eq:nu_n_int}
w_n(x)= \int_0^{\infty} y^{-1} \wbar_n(xy)\mathfrak{z}\left(\frac1y\right)dy
\end{equation}
for a.e. $x>0$. In other words, we have, with the obvious notation, $w_n \eqae \wbar \surd \mathfrak{z}$ where $\surd$ represents the Mellin convolution, see \cite[Section 11.11]{Misra_1986}. Therefore, by \cite[(11.11.4)]{Misra_1986}, we have, for any $\Re(z)> n$,
\[\M_{w_n}(z)=\M_{\Z}(z)\M_{\wbar_n}(z)=\M_{\Z}(z)\frac{(-1)^n}{n!}\frac{\Gamma(z)}{\Gamma(z-n)}\M_{V_{\T_1\psi}}(z)=\frac{(-1)^n}{n!}\frac{\Gamma(z)}{\Gamma(z-n)}\M_{V_{\psi}}(z)\]
where the last identity comes from the factorization \eqref{decomp_V1Z}. Observe that the right-hand side of the above equation is indeed the Mellin transform of $\mathcal{R}^{(n)}\frakm(x)$, and by injectivity of the Mellin transform, we conclude that $w_n(x)\eqae \mathcal{R}^{(n)}\frakm(x)$, or equivalently
\[\frakm_n(x)= \frac{\mathcal{R}^{(n)}\frakm(x)}{\frakm(x)}\]
for almost every $x>0$, which can be extended to every $x>0$ by the continuity of $\frakm_n$ and the smoothness of $\frakm$, see Corollary~\ref{cor:supp_smoo_m}. Furthermore, by the intertwining relationship \eqref{eq:Intertwin_P_P1},
\begin{equation} \label{Intertwin3_dual}
\hatP_t\frakm_n(x)=\hatP_t\widehat{\Lambda}_{\Z} \breve{\frakm}_n(x) =\widehat{\Lambda}_{\Z}\widehat{\ttP}_t \breve{\frakm}_n(x)=e^{-nt}\widehat{\Lambda}_{\Z} \breve{\frakm}_n(x)=e^{-nt}\frakm_n(x),
\end{equation}which shows that $\frakm_n$ is an eigenfunction for $\hatP$ (or co-eigenfunction for $P$). Finally, take any $g\in \rmL^2(m)$, then by the co-eigenfunction property of $\frakm_n$ and the intertwining relation \eqref{eq:intertwin_refl_ss}, we have
\begin{align*}
e^{-nt}\left\langle \hatIp\frakm_n,g\right\rangle_{m} &= e^{-nt}\left\langle \frakm_n,\Ip g\right\rangle_{\frakm}=\left\langle \hatP_t\frakm_n,\Ip g\right\rangle_{\frakm}=\left\langle \frakm_n,P_t\Ip g\right\rangle_{\frakm}\\
& = \left\langle \frakm_n,\Ip Q_tg\right\rangle_{\frakm}=\left\langle \hatIp\frakm_n,Q_tg\right\rangle_{m}.
\end{align*}
In other words, $\hatIp\frakm_n$ is a co-eigenfunction of $Q_t$, which is indeed $\Lag_n$ since $Q_t$ is self-adjoint. Moreover, recalling that $\Ip$ has a dense range in $\Lnu$, we have that $\hatIp$ is one-to-one on $\Lnu$ and thus equation $\hatIp f = \Lag_n$ has at most one solution in $\Lnu$, which is indeed $\frakm_n$. Therefore, we deduce that, for any $m,n\geq 0$,
\begin{equation} \label{eq:biortho}
 \left\langle \Poly_m,\frakm_n\right\rangle_{\frakm}= \frakc_m(-\theta) \left\langle \Ip \Lag_m,\frakm_n\right\rangle_{\frakm} = \frakc_m(-\theta) \left\langle \Lag_m,\hatIp\frakm_n\right\rangle_{m} =  \frakc_m(-\theta) \left\langle \Lag_m,\Lag_n\right\rangle_{m} = \mathbf{1}_{\{m=n\}},
 \end{equation}
by the orthogonality property of the Laguerre polynomials. This shows that the sequences $(\Poly_n)_{n\geq 0}$ and $(\frakm_n)_{n\geq 0}$ are biorthogonal. Next, by \cite{CKP_transformation}, $\T_1\psi$ and $\psi$ have the same parameter $\sigma^2$, hence $\psi\in \N_{P}\cap \N_{\checkmark}$ if and only if $\T_1\psi \in \N_{P}\cap \N_{\uparrow} $. Moreover, observing that $\phi(\infty)=\phi_1(\infty)=\beta+\overline{\overline{\Pi}}(0+)$, hence  by \cite[Theorem 9.1 and Theorem 10.1]{PS_spectral}, the bounds on the right-hand side of \eqref{eq:nu_n_estimate_NP} and \eqref{eq:nu_n_estimate_frakm} hold for $\|\breve{\frakm}_n\|_{\breve{\frakm}}$. Since $\frakm_n=\widehat{\Lambda}_{\Z}\breve{\frakm}_n$ and $|||\widehat{\Lambda}_{\Z}||| = |||\Lambda_{\Z}|||\leq 1$,  we conclude the same bounds for $\|\frakm_n\|_{\frakm}$. Finally, by \cite[Theorem 10.1]{PS_spectral}, the sequence $(\sqrt{\frakc_n(\frakb)}\breve{\frakm}_n)_{n\geq 0}$ is a Bessel sequence in $\rmL^2(\breve{\frakm})$ with bound 1, hence we have, for any $f\in \Lnu$,
\begin{align*}
\sum_{n=0}^{\infty}\left|\left\langle f,\sqrt{\frakc_n(\frakb)}\frakm_n\right\rangle_{\frakm}\right|^2=\sum_{n=0}^{\infty}\left|\left\langle f,\sqrt{\frakc_n(\frakb)}\widehat{\Lambda}_{\Z} \breve{\frakm}_n\right\rangle_{\frakm}\right|^2=\sum_{n=0}^{\infty}\left|\left\langle \Lambda_{\Z} f,\sqrt{\frakc_n(\frakb)}\breve{\frakm}_n\right\rangle_{\breve{\frakm}}\right|^2\leq\| \Lambda_{\Z} f \|^2_{\breve{\frakm}}\leq \|f\|^2_{\frakm}
\end{align*}
since $|||\Lambda_{\Z}|||\leq 1$. This proves that $(\sqrt{\frakc_n(\frakb)} \frakm_n)_{n\geq 0}$ is a Bessel sequence in $\Lnu$. Now in the case of $\frakm^{\dag}_n$, let us first prove that it is in $\Lnu$, which suffices to show its $\Lnu$-integrability around the neighborhoods of 0 and infinity. To this end, define $d_{\phi_1}=\sup\{u<0; \phi_1(u)=-\infty \textnormal{ or } \phi_1(u)=0\}$, where we recall that $\phi_1(u)=\frac{\mathcal{T}_1\psi(u)}{u}=\frac{\psi(u+1)}{u+1}$, then we easily observe that $d_{\phi_1}=\theta-1$ since $\theta$ is the largest root of $\psi$. Hence by combining \cite[Theorem 5.4]{PS_spectral}  and \eqref{eq:relation_nu_nubar}, we see that for any $a>\theta$ and $A\in (0,\frakr)$, that exists a constant $C_{a,A}>0$ such that $\frakm(x)\geq C_{a,A}x^{a}$ for all $x\in (0,A)$. Therefore, denoting $w^{\dag}_n=\frakm^{\dag}_n \frakm$, then we see that
\[(\frakm^{\dag}_n(x))^2\frakm(x)=\frac{(w^{\dag}_n(x))^2}{\frakm(x)} \leq \frac{1}{C_{a,A}} x^{-a}(w^{\dag}_n(x))^2\]
for all $x\in (0,A)$. Hence to prove the $\Lnu$-integrability of $\frakm^{\dag}_n$ around 0, it suffices to prove the $\rmL^2(p_{-a})$-integrability of $w^{\dag}_n$ around 0, where $p_{-a}(x)dx=x^{-a}dx$. However, observe that $w^{\dag}_n=\frac{\Rcal^{(n)}m^{\uparrow}}{p_{\theta}}$, thus by taking the Mellin transform on both sides, we have, for $\Re(z)>n+\theta$,
\[\M_{w^{\dag}_n}(z)=\M_{\Rcal^{(n)}m^{\uparrow}}(z-\theta)=\frac{(-1)^n}{n!}\frac{\Gamma(z-\theta)}{\Gamma(z-\theta-n)}W_{\phi^{\uparrow}}(z-\theta)=\frac{(-1)^n}{n!}\frac{\Gamma(z-\theta)}{\Gamma(z-\theta-n)}\frac{W_{\phi}(z)}{W_{\phi}(1+\theta)},\]
where for the last identity we used \cite[(8.12)]{PS_spectral}, with $\phi^{\uparrow}(u)=\frac{\psi_{\uparrow}(u)}{u}=\phi(u+\theta)$. Therefore, using the Stirling approximation \eqref{eq:stir} as well as the asymptotic behavior of $W_{\phi}$ by \cite[Theorem 5.1(3)]{PS_spectral}, we have, for large $|b|$, that
\begin{equation}
\M_{p_{-\frac{a}{2}}w^{\dag}_n}\left(\frac12+ib\right)=\M_{w^{\dag}_n}\left(\frac{1-a}{2}+ib\right)={\rm{o}}\left(|b|^{n-u}\right)
\end{equation}
for some $u>n+\frac12$. Hence $b\mapsto \M_{p_{\frac{-a}{2}}w^{\dag}_n}\left(\frac12+ib\right) \in \rmL^2(\R)$, and $x\mapsto x^{-\frac{a}{2}}w^{\dag}_n(x) \in \rmL^2(\R_+)$ by the Parseval identity of Mellin transform, that is $w^{\dag}_n \in \rmL^2(p_{-a})$. This proves the $\Lnu$-integrability of $\frakm^{\dag}_n$ around 0. On the other hand, since $\M_{m^{\uparrow}}(u)=W_{\phi^{\uparrow}}(u)=\frac{W_{\phi}(u+\theta)}{W_{\phi}(1+\theta)}$, we have \[\M_{p_{\theta}\frakm}(u)=\M_{\frakm}(u+\theta)=\frac{\Gamma(u)}{\Gamma(u+\theta)\Gamma(1-\theta)}W_{\phi}(u+\theta)=C \M_{B(1,\theta)}(u)\M_{m^{\uparrow}}(u),\]
where $C=\frac{W_{\phi}(1+\theta)}{\Gamma(1-\theta)\Gamma(1+\theta)}$ and $B(1,\theta)$ is a Beta distribution of parameter $(1,\theta)$. Hence by the formula for the density of product of random variables, we have, for $x$ large enough such that $m^{\uparrow}$ is non-increasing on $(x,\infty)$,
\begin{eqnarray*}
\frac{1}{C}\frakm(x)p_{\theta}(x)&=& \int_x^{\infty}m^{\uparrow}(y)\left(1-\frac{x}{y}\right)^{\theta-1}\frac1y dy = \int_x^{\infty}y^{-\theta}m^{\uparrow}(y)(y-x)^{\theta-1}dy \\
& \geq & \int_x^{x+1}y^{-\theta}m^{\uparrow}(y)(y-x)^{\theta-1}dy \geq (x+1)^{-\theta}m^{\uparrow}(x+1) \geq C_{\psi}x^{-\theta}m^{\uparrow}(x)
\end{eqnarray*}
for some $C_{\psi}>0$ by \cite[Theorem 5.5 (1)]{PS_spectral}. Combine the above relations together, we have, for $x$ large enough,
\[\frac{m^{\uparrow}(x)}{x^{2\theta}\frakm(x)} \leq \frac{1}{C C_{\psi}}.\]
Now denoting $m^{\uparrow}_n = \frac{\Rcal^{(n)}m^{\uparrow}}{m^{\uparrow}}$, which is in $\rmL^2(m^{\uparrow})$ by \cite[Theorem 8.1]{PS_spectral}, then we have $(\frakm^{\dag}_n(x))^2\frakm(x) = (m^{\uparrow}_n(x))^2 m^{\uparrow}(x)\frac{m^{\uparrow}(x)}{x^{2\theta}\frakm(x)} \leq \frac{1}{C C_{\psi}} (m^{\uparrow}_n(x))^2 m^{\uparrow}(x)$ and is integrable around $\infty$. Hence $\frakm^{\dag}_n\in \Lnu$ for all $n\in\mathbb{N}$. Furthermore, again by \cite[Theorem 8.1]{PS_spectral}, $m^{\uparrow}_n$ is the co-eigenfunction for $P^{\uparrow}_t$ with eigenvalue $e^{-nt}$. Hence we have, for any $n\in\mathbb{N}$,
\begin{eqnarray*}
\left\langle P^{\dag}_tf,\frakm^{\dag}_n\right\rangle_{\frakm} &=& e^{-\theta t}\left\langle p_{\theta}P^{\uparrow}_t \frac{f}{p_{\theta}},\frac{\Rcal^{(n)}m^{\uparrow}}{p_{\theta}\frakm}\right\rangle_{\frakm} = e^{-\theta t}\left\langle P^{\uparrow}_t \frac{f}{p_{\theta}},m^{\uparrow}_n\right\rangle_{m^{\uparrow}} \\
&=& e^{-(n+\theta) t}\left\langle \frac{f}{p_{\theta}},m^{\uparrow}_n\right\rangle_{m^{\uparrow}}=e^{-(n+\theta) t}\left\langle f,\frac{m^{\uparrow}_nm^{\uparrow}}{p_{\theta}\frakm}\right\rangle_{\frakm}=e^{-(n+\theta) t}\left\langle f,\frakm^{\dag}_n \right\rangle_{\frakm}.
\end{eqnarray*}
Therefore $\frakm^{\dag}_n$ is a co-eigenfunction for $P^{\dag}_t$ with eigenvalue $e^{-(n+\theta)t}$. On the other hand, any solution $f$ of the equation $\hatIp f = \Lag_n^{\dag}$ shall satisfy the relation
\[\frac{\Gamma(1-\theta)}{W_{\phi}(1+\theta)}m(x)\Lag^{\dag}_n(x) \eqae \int_0^{\infty}y^{-1}f(xy)\frakm(xy)\iota\left(\frac1y\right)dy.\]
Hence taking Mellin transform on both sides and after some careful computations, we have
\[\M_{\frakm f}(u)=\frac{(-1)^n}{n!}\frac{\Gamma(u-\theta)}{\Gamma(u-\theta-n)}\frac{W_{\phi}(u)}{W_{\phi}(1+\theta)}=\M_{w^{\dag}_n}(u).\]
Therefore we see that $\frakm^{\dag}_n$ is a solution of $\hatIp f = \Lag_n^{\dag}$ by injectivity of the Mellin transform, and the uniqueness of this solution is due to the one-to-one property of $\hatIp$. Hence the biorthogonality of $(\Poly_n^{\dag},\frakm_n^{\dag})_{n\geq 0}$ follows by a similar argument as \eqref{eq:biortho}. This completes the proof.
\subsubsection{Proof of Theorem~\ref{thm:spectral}\eqref{it:spec}}
%Having characterized a sequence of eigenfunctions and coeigenfunctions, we are now ready to prove Theorem~\ref{thm:spectral}\eqref{it:spec}.
First, take any $f\in Ran(\Ip)$ with $\Ip g = f$ for some $g\in\Lg$, then by the intertwining relation \eqref{eq:intertwin_refl_ss} and the spectral expansion for $Q_t$, see \eqref{eq:spectral_Q}, we have
\begin{align*}
P_tf(x)=P_t\Ip g(x)=\Ip Q_tg(x)=\Ip \sum_{n\geq 0}e^{-nt}\frakc_n(-\theta)\left\langle g,\Lag_n  \right\rangle_{m}\Lag_n(x)=\sum_{n\geq 0}e^{-nt}\left\langle g,\Lag_n  \right\rangle_{m}\Poly_n(x),
\end{align*}
where the last identity is justified by the fact that $\Ip \in \B(\Lg,\Lnu)$, the Bessel property of $\left(\frakc^{-\frac{1}{2}}_n(-\theta)\Poly_n\right)_{n\geq 0}$ combined with the fact that the sequence $\left(\sqrt{\frakc_n(-\theta)}e^{-nt}\left\langle g,\Lag_n  \right\rangle_{m}\right)_{n\geq 0} \in \ell^2$ since $\left(\left\langle g,\Lag_n  \right\rangle_{m}\right)_{n\geq 0}\in \ell^2$. %Next, from Proposition~\ref{proposition_multip_kernel}, $\overline{Ran}(\Ip)=\Lnu$, hence we have that the pseudo-inverse $\Ip^{\dagger}$ is densely defined from $\Lnu$ to $\Lg$ and thus for any $f\in Ran(\Ip)$ and $t,x>0$,
%\begin{equation*}
%P_tf=\sum_{n=0}^{\infty} e^{-nt}\left\langle \Ip^{\dagger} f,\Lag_n \right\rangle_{m}\Poly_n(x)
%\end{equation*}
%in $\Lnu$.
Moreover, recalling that $\hatIp \frakm_n = \Lag_n$, we see that $\left\langle g,\Lag_n\right\rangle_{m}=\left\langle \Ip g,\frakm_n\right\rangle_{\frakm}=\left\langle f,\frakm_n\right\rangle_{\frakm}$, hence this proves \eqref{eq:spectral_P} for all $(\psi,f)\in\D^{\checkmark}(\Ip)$.
%we first observe that by \eqref{eq:I_star_nu_n}, for any $f\in Ran(\Ip)$, we have
%\begin{equation}\label{eq:inner_prod_equal_coeigen}
%\left\langle \Ip^{\dagger}f,\Lag_n\right\rangle_{m} = \left\langle f,\frakm_n\right\rangle_{\nu}.
%\end{equation}
Now let us define the spectral operator $S_t$,  $t\geq 0$, by
\begin{equation}
S_tf(x)=\sum_{n=0}^{\infty}e^{-nt}\left\langle f,\frakm_n\right\rangle_{\frakm}\Poly_n(x).
\end{equation}
We first note that under the condition $\D^{P}(\frakm)$,
\begin{equation*}
\sqrt{\frakc_n(-\theta)}e^{-nt}\left\langle f,\frakm_n\right\rangle_{\frakm} \leq e^{-nt}\|f\|_{\frakm}\left\|\frakm_n\right\|_{\frakm}=O\left(n^{\frac{\theta}{2}}e^{(-t+\epsilon)n}\right).
\end{equation*}
Hence $(\sqrt{\frakc_n(-\theta)}e^{-nt}\left\langle f,\frakm_n\right\rangle_{\frakm})_{n\geq 0} \in \ell^2$. By the Bessel property of the sequence $\left(\frakc^{-\frac{1}{2}}_n(-\theta)\Poly_n\right)_{n\geq 0}$, we get that $S_tf(x)\in \Lnu$ for $(\psi,f)\in \D^{\checkmark}(\Ip) \cup \D^{\N_P}(\frakm)$. Our next aim is to show $P_tf(x)=S_tf(x)$ under the conditions $\D^{\N_P}(\frakm) \backslash \D^{\checkmark}(\Ip)$. Since $Ran(\Ip)$ is dense in  $\Lnu$, for any $f\in \Lnu$, there exists a sequence $(g_m)_{m\geq 0} \in \Lg$ such that $\lim_{m\rightarrow\infty}\Ip g_m =f$ in $\Lnu$. Hence we have from the previous part that
\begin{equation*}
P_t\Ip g_m(x)=\sum_{n=0}^{\infty}c_{n,t}(\Ip g_m)\frakc^{-\frac{1}{2}}_n(-\theta)\Poly_n(x),
\end{equation*}
where the constants $c_{n,t}$ are defined by $c_{n,t}(f)=\sqrt{\frakc_n(-\theta)}e^{-nt}\left\langle f,\frakm_n \right\rangle_{\frakm}$ for $f\in\Lnu$. Now let us define operator $\mathcal{S}:\ell^2 \rightarrow \Lnu$ by, for any $(c_n)_{n\geq 0}\in \ell^2$,
\begin{equation} \label{eq:defn_S}
\mathcal{S}((c_n))=\sum_{n=0}^{\infty}c_n \frakc^{-\frac{1}{2}}_n(-\theta)\Poly_n.
\end{equation}
Then by \cite[(2.5)]{PS_spectral}, $\mathcal{S}$ is a bounded operator with operator norm $|||\mathcal{S}|||$ and
\begin{equation*}
\|P_t\Ip g_m-S_tf\|_{\frakm}^2 = \|\mathcal{S}(c_{n,t}(\Ip g_m-f))\|_{\frakm}^2\leq |||\mathcal{S}||| \sum_{n=0}^{\infty}c_{n,t}^2(\Ip g_m-f)\leq C_t\|\Ip g_m-f\|_{\frakm}^2
\end{equation*}
for some constant $0<C_t<\infty$. Hence $\lim_{m\rightarrow\infty}P_t\Ip g_m=S_tf$. However, since $P_t$ is a contraction, we conclude that $P_tf=S_tf$ under $\D^{\N_P}(\frakm)$. The spectral expansion of $P^{\dag}_tf$ for $(\psi,f)\in \D^{\checkmark}(\Ip)$ can be proved similarly using the spectral expansion of $Q^{\dag}_tf$ in \eqref{eq:spectral_Q_dag}, the intertwining between $P^{\dag}$ and $Q^{\dag}$, and the properties of $\Poly^{\dag}_n$ as well as $\frakm^{\dag}_n$. Finally, for $(\psi,f)\in \D^{\N_P}(\frakm)$, we have $\psi_{\uparrow}\in \N_P\cap \N_{\uparrow}$ and therefore by \cite[Theorem 1.11]{PS_spectral}, for all $f\in \rmL^2(m^{\uparrow})$,
\[P^{\uparrow}_tf = \sum_{n=0}^{\infty}e^{-nt}\left\langle f,m^{\uparrow}_n\right\rangle_{m^{\uparrow}} \Poly^{\psi_{\uparrow}}_n.\]
Hence
\begin{equation*}
P^{\dag}_tf = e^{-\theta t}p_{\theta}P^{\uparrow}_t\left(\frac{f}{p_{\theta}}\right)= \sum_{n=0}^{\infty}e^{-(n+\theta)t}\left\langle \frac{f}{p_{\theta}},m^{\uparrow}_n\right\rangle_{m^{\uparrow}} \Poly^{\dag}_n=\sum_{n=0}^{\infty}e^{-(n+\theta)t}\left\langle f,\frakm^{\dag}_n\right\rangle_{\frakm} \Poly^{\dag}_n.
\end{equation*}
This completes the proof of Theorem~\ref{thm:spectral}.

\subsubsection{Proof of Corollary~\ref{thm:convergence}}
%First, for any $\psi\in\N_{\checkmark}$, $\mathfrak{f}\in \Lg$ and $t>0$,
%\begin{align*}
%\|P_t \Ip  \frakf -\frakm \Ip \frakf\|_{\frakm}^2 &= \|\Ip Q_t\frakf - m \frakf\|_{\frakm}^2\\
%&\leq \sum_{n=1}^{\infty}\left|\left\langle Q_t\frakf,\Lag_n\right\rangle_{m}\right|^2=\sum_{n=1}^{\infty}\left|\left\langle Q_t\frakf-m\frakf,\Lag_n\right\rangle_{m}\right|^2\\
%&=\|Q_t\frakf-m\frakf\|_{m}^2 \leq e^{-2t}\|\frakf-m\frakf\|_{m}^2,
%\end{align*}
%where we have successively used the intertwining relationship \eqref{eq:intertwin_refl_ss}, the Parseval identity for the Laguerre polynomials and the exponential decay of the Laguerre semigroup. On the other hand, when
For any $\psi\in \N_P\cap \N_{\checkmark}$ and assuming $\overline{\overline{\Pi}}(0+)<\infty $, since by Theorem~\ref{thm:spectral}, $\left(\frakc^{-\frac{1}{2}}_n(-\theta)\Poly_n\right)_{n\geq 0}$ and $(\sqrt{\frakc_n(\frakb)}\frakm_n)_{n\geq 0}$ are both Bessel sequences in $\Lnu$ with bound 1, we have, for $t>T_{\frakb}=\frac12\ln\left(\frac{\frakb+2}{2-\theta}\right)$,
\begin{align*}
\|P_t f-\frakm f\|_{\frakm}^2&= \|\mathcal{S}(c_{n,t}(f))\|_{\frakm}^2 \leq \sum_{n=1}^{\infty} \frac{\frakc_n(-\theta)}{\frakc_n(\frakb)}\left|\left\langle P_tf,\sqrt{\frakc_n(\frakb)}\frakm_n \right\rangle_{\frakm}\right|^2 \\
&= e^{-2t}\sum_{n=1}^{\infty} \frac{e^{-2(n-1)t}\frakc_n(-\theta)}{\frakc_n(\frakb)}\left|\left\langle f,\sqrt{\frakc_n(\frakb)}\frakm_n \right\rangle_{\frakm}\right|^2\\
&=\frac{e^{-2t}\frakc_1(-\theta)}{\frakc_1(\frakb)}\sum_{n=1}^{\infty} \frac{e^{-2(n-1)t}\frakc_1(\frakb)\frakc_n(-\theta)}{\frakc_n(\frakb)\frakc_1(-\theta)}\left|\left\langle f-\frakm f,\sqrt{\frakc_n(\frakb)}\frakm_n \right\rangle_{\frakm}\right|^2\\
&\leq \frac{\frakb+1}{1-\theta}e^{-2t}\sum_{n=1}^{\infty} \left|\left\langle f-\frakm f,\sqrt{\frakc_n(\frakb)}\frakm_n \right\rangle_{\frakm}\right|^2\\
&\leq \frac{\frakb+1}{1-\theta}e^{-2t}\|f-\frakm f\|_{\frakm}^2,
\end{align*}
where we used the fact that by the Stirling approximation, $\frac{e^{-2(n-1)t}\frakc_1(\frakb)\frakc_n(-\theta)}{\frakc_n(\frakb)\frakc_1(-\theta)}\leq 1$ for all $t>T_{\frakb}$. On the other hand, for $t\leq T_{\frakb}$, $ \frac{\frakb+1}{1-\theta}e^{-2t} \geq \frac{\frakb+1}{\frakb+2}\frac{2-\theta}{1-\theta}\geq 1$ since $\frakb\geq 0 > -\theta$. Invoking that $P_t$ is a contraction, this concludes the proof of this corollary.

\bibliographystyle{plain}
%\bibliography{Untitled}

\end{document}